\documentclass[12pt,final]{amsart}
\usepackage[margin=1in,includeheadfoot]{geometry}
\usepackage{amsmath}

\usepackage{graphicx}
\usepackage{stmaryrd}
\usepackage{color}
\usepackage[dvipsnames]{xcolor}

\usepackage[colorlinks=true,
            linkcolor=violet,
            urlcolor=Aquamarine,
            citecolor=YellowOrange]{hyperref}

\newtheorem{thm}{Theorem}[section]

\newtheorem{lem}[thm]{Lemma}
\newtheorem{prop}[thm]{Proposition}
\newtheorem{defn}[thm]{Definition}
\newtheorem{rem}[thm]{Remark}
\newtheorem{que}[thm]{Question}
\newtheorem{exam}[thm]{Example}

\definecolor{darkgreen}{rgb}{0,0.5,0}
\definecolor{darkred}{rgb}{0.7,0,0}

\def\ba{\begin{array}}
\def\ea{\end{array}}
\def\be{\begin{equation}}
\def\ee{\end{equation}}
\def\bee{\begin{eqnarray}}
\def\beee{\begin{eqnarray*}}
\def\eee{\end{eqnarray}}
\def\eeee{\end{eqnarray*}}
\def\nn{\nonumber}

\title[Yang-Mills energy quantization]{%\sc
\bf Yang-Mills energy quantization over non-collapsed degenerating Einstein manifolds and applications}
\thanks{The second author would like to thank Professor Aaron Naber for valuable comments on the theory of non-collapsed Einstein manifolds.}

\author{Youmin Chen}
\address{Department of Mathematics, Shantou University\\ 5 Cuifeng Road \\ Shantou, Guangdong, 515063 \\P. R. China}%
\email{youminchen@stu.edu.cn}%
\author{Miaomiao Zhu}
\address{School of Mathematical Sciences, Shanghai Jiao Tong University\\ 800 Dongchuan Road \\ Shanghai, 200240 \\P. R. China}%
\email{mizhu@sjtu.edu.cn}%
\date{\today}

\begin{document}
\maketitle
\begin{abstract}
We investigate a sequence of Yang-Mills connections $A_j$ lying in vector bundles $E_j$ over non-collapsed degenerating closed Einstein 4-manifolds $(M_j, g_ j)$ with uniformly bounded Einstein constants and bounded diameters. We establish a compactness theory modular three types of bubbles. As applications, we get some quantization results for several important topological number associated with the vector bundles, for instance, the first Pontrjagin numbers $p_1(E)$ of vector bundles over Einstein 4-manifolds and the Euler numbers $\chi(M;E)$ of holomorphic vector bundles over K\"{a}hler-Einstein surfaces. Furthermore, we get some quantization results about the volume $v(L_j)$ and certain cohomological numbers (e.g. $dim H^0(M_j;L_j)$) of holomorphic line bundles $L_j$ over non-collapsed degenerating K\"{a}hler-Einstein surfaces $(M_j,J_j,g_j)$ with the aid of the classical vanishing theorems, the classical Hirzebruch-Riemann-Roch type theorems, and the profound convergence theory of K\"{a}hler-Einstein manifolds. In particular, we obtain some interesting identities involving non-collapsed degenerating compact K\"{a}hler-Einstein surfaces with non-zero scalar curvature, which indicate that we can know the Euler number of $M_j$ for large $j$ provided some topological information of the limit orbifold $M_\infty$. For K\"{a}hler-Einstein Del Pezzo surfaces, an interesting implication is that we can provide some preliminary estimates for the number of singularities of various types in $M_\infty$ in an effective way. As an unexpected surprise, we find an identity which connects Milnor numbers for singularities in $M_\infty$ and the correction terms in the Hirzebruch-Riemann-Roch theorem for orbifolds.
Some results like the compactness modular bubbles, the quantization about the Euler numbers of holomorphic vector bundles and the volume of holomorphic line bundles can be extended to the case of higher dimensional $n$-manifolds by imposing some further $L^{\frac{n}{2}}$ bound assumptions on Riemannian curvature and Yang-Mills curvature.
\end{abstract}

\

\section{Introduction}

\vskip5pt

Yang-Mills theory is important in both mathematics and physics. The space of Yang-Mills connections has been investigated extensively in various situations, for instance, anti-self-dual connections on compact 4-manifolds e.g. \cite{Taubes1982Self,Taubes1984Self,geofourfolds,gaugefourtopology},  anti-self-dual connections on asymptotically periodic 4-manifolds e.g. \cite{GOMPF85,Taubes1987jdg}, self-dual connections on orbifolds e.g. \cite{freeorbifolds,1986Compactness}, Yang-Mills theory over quasiconformal 4-manifolds \cite{DS1989} and Yang-Mills instantons on gravitational instantons e.g. \cite{Kronheimer1990Yang,Nakajima1990}.

Since the analysis foundation of Yang-Mills theory was established in \cite{UhlenbeckRemovablesingu, UhlenbeckLp}, various geometric analysis aspects of Yang-Mills fields has been extensively studied. In particular, the energy quantization for Yang-Mills fields over a fixed 4-manifold  was developed in e.g. \cite{geofourfolds, gaugefourtopology, DS1989, Instanton4manifolds, RiviequantYM} and the case of higher dimensions was established in \cite{RiviequantYM,tian2000y, Naberyangenergy}.

In this paper, we shall firstly investigate the compactness problem for the moduli space of Yang-Mills fields over non-collapsed varying Einstein 4-manifolds, and then explore some applications to the geometry and topology of Einstein 4-manifolds. In a previous work \cite{CZdegebiharm}, we established the energy quantization and bubble tree convergence for biharmonic maps from non-collapsed degenerating Einstein 4-manifolds into compact Riemannian manifolds.

Let $A$ be a connection in a vector bundle $E$ over a Riemannian manifold $(M,g)$ with compact structure group $G\subset SO(r)$ for some $r\in \mathbb{Z}_{+}$. A Yang-Mills connection $A$ is a critical point of the Yang-Mills energy functional
\begin{equation}\label{ymfunct}
\mathcal{YM}(A)=\int_M| F_A|^2 dV_g,
\end{equation}
where $F_A$ is the curvature form. The Euler-Lagrange equation is
\begin{equation}\label{ymequintro}
D_{A}^{*}F_A=0.
\end{equation}

Let $(M_j ,g_j)$ be a sequence of closed Einstein 4-manifolds with uniformly bounded Einstein constants $|\mu_j|\leq \mu$ and satisfying
\begin{equation}\label{einstincondintr}
diam(M_j ,g_j)\leq D, \,\, vol(M_j ,g_j)\geq V
\end{equation}
for some positive constants $D>0, V>0$.
By the profound work \cite{cheeger2015regularity} (see also \cite{jiangnaber21} for higher dimensional case and \cite{CJN2021,Anderson94icm,CCT2002,Cheeger2003gafa,ct05cmp} for related works),
$(M_j ,g_j)$ have a priori $L^2$ Riemannian curvature estimates
\begin{equation}\label{cheegernaberjiang}
\int_{M_j}|Rm(g_j)|^2 dV_{g_j}\leq R(\mu,D,V),
\end{equation}
where $Rm(g)$ denotes the Riemann curvature tensor of the metric $g$, and $R(\mu,D,V)$ is a positive constant depending on $\mu, \,D$ and $V$. Let $A_j$  be a sequence of Yang-Mills connections in vector bundles $E_j$ with structure group $G$ over the Einstein 4-manifolds $(M_j, g_j)$ such that the Yang-Mills energy is uniformly bounded, namely,
\begin{equation}\label{ymbounds}
\mathcal{YM}(A_j)\equiv\int_{M_j}|F_{A_j}|^2 dV_{g_j}\leq C
\end{equation}
for some finite $C> 0$.

By the classical compactness theory of non-collapsed Einstein manifolds e.g. \cite{nakajima1988hausdorff,anderson1989ricci, BKN, Tian1990,  bando1990bubbling, Anderson1992,Anderson94icm,  nakajima1994convergence,Jeff2006Curvature,cheeger2015regularity}, we know that the underlying Einstein 4-manifolds $(M_j ,g_j)$ possesses a nice bubble tree convergence (up to a subsequence), where the bubbles are Ricci flat ALE (\textit{Asymptotically Locally Euclidean}) manifolds (orbifolds), see Theorem \ref{mainconvgethm} in Section \ref{Preliminaries111} for a detailed description. Then through detailed blow-up analysis of Yang-Mills connections $A_j$ in $E_j$ over $M_j$ (see Section \ref{treeidentity}), we get that $A_j$ converges (up to a subsequence still denoted by $A_j$) to a limit Yang-Mills connection
$A_\infty$ in some vector bundle $E_\infty$ over $M_\infty$ modular finitely many Yang-Mills bubble connections.

On one hand, the sequence of Yang-Mills connections $A_j$ may blow up at some points $b_l, l=1, \cdots, v$ which are away from the orbifold singularities of $M_\infty$, and the corresponding bubble connections are Yang-Mills connections in bundles over $\mathbb{R}^4$, denoted by $A_{\omega_{b_l,k}}, k=1,\cdots, N_{b_l}$.

On the other hand, the sequence of $A_j$ can possibly blow up at some orbifold singularities $a_l, l=1,\cdots, u$ of $M_{\infty}$, then, in general, we may obtain the following three types of bubble connections:

\

\begin{itemize}
\item[{\bf Type $\alpha$}:] Yang-Mills connections in bundles over $\mathbb{R}^4$, denoted by $A_{\omega^{a_l,\alpha,k}},  \ k=1,\cdots, N_{a_l,\alpha}$; \\
\item[{\bf Type $\beta$}:] Yang-Mills connections in bundles over some complete, Ricci flat, non-flat ALE  4-manifolds (orbifolds), denoted by $A_{\omega^{a_l,\beta,k}}$, $\ k=1,\cdots, N_{a_l,\beta}$;\\
\item[{\bf Type $\gamma$}:] Yang-Mills connections in bundles over $\mathbb{R}^4 / \Gamma$ for some nontrivial finite group $\Gamma \subset SO(4)$, denoted by $A_{\omega^{a_l,\gamma,k}},  \ k=1,\cdots, N_{a_l,\gamma}$.\\
\end{itemize}

Now we state the energy quantization result for Yang-Mills connections over non-collapsed degenerating Einstein 4-manifolds.

\begin{thm}\label{degeymenergyintr} Let $(M_j,g_j, E_j, A_j)$ be as above. Then we have the following identity:

\begin{eqnarray*}
\lim_{j\rightarrow \infty}\int_{M_j}|F_{A_j}|^2 dV_{g_j} &=& \int_{M_\infty}|F_{A_\infty}|^2dV_{g_\infty} +\sum_{l=1}^{v}\sum_{k=1}^{N_{b_l}}\mathcal{YM}(\mathbb{R}^4, A_{\omega_{b_l,k}})  \\
&& +\sum_{l=1}^{u}\sum_{\eta=\alpha,\beta,\gamma}\sum_{k=1}^{N_{a_l,\eta}} \mathcal{YM}(X_{\eta,k,a_l}, A_{\omega^{a_l,\eta,k}}),
\end{eqnarray*}
where $\mathcal{YM}(X, A)$ is the Yang-Mills energy of the connection $A$ over the underlying space $X$. Here $X_{\alpha,k,a_l}$ is $\mathbb{R}^4$, $X_{\beta,k,a_l}$ is a Ricci flat ALE bubble space and $X_{\gamma,k,a_l}$ is of the form $\mathbb{R}^4/\Gamma$ for some nontrivial finite group $\Gamma \subset SO(4)$.
\end{thm}

As an application of Theorem \ref{degeymenergyintr}, we shall derive a quantization result for the first Pontrjagin number $p_1(E)$ for vector bundles over a sequence of non-collapsed Einstein 4-manifolds (see Theorem \ref{degeymenergycor}) provided the existence for Yang-Mills connections.
It is well known that the general existence of Yang-Mills connections is a hard problem, however, there has been a lot of progress in this field, to name a few, the existence of (anti-)self-dual connections \cite{AHDM1978,AtiyahSelf,Taubes1982Self,Taubes1984Self,geofourfolds,1985Anti,UY1986,Donaldson87bundle} etc, and the existence of non-minimal Yang-Mills connections \cite{Uhlenbeck1989,SadunSegert92,parkerinvention92} etc.

Furthermore, we get a quantization result for the Euler numbers of holomorphic vector bundles $\chi(M;E)$ over K\"{a}hler-Einstein surfaces (see Theorem \ref{appli33}) by applying Theorem \ref{degeymenergyintr} and the classical Riemannian curvature energy identity for non-collapsed Einstein 4-manifolds(real dimension) (see Theorem \ref{mainconvgethm}). Note that the quantity $\chi(M; E)$ is closely related to the famous Hirzebruch-Riemann-Roch theorem. This quantization result is widely applicable. On one hand, so far we have gained a deep understanding about the existence of K\"{a}hler-Einstein metrics on closed K\"{a}hler manifolds by magnificent work in the last decades, see \cite{yau78calabi} for the existence of Ricci flat K\"{a}hler metrics, see \cite{AubinEinstein76,yau78calabi,caoeinstein85} for the existence of K\"{a}hler-Einstein metrics with negative scalar curvature, and see e.g.
\cite{yauopenprob,tianyau87,Tian1990,tianinvt1997,Donaldson2002Scalar,CDS1,CDS2,CDS3,tiancpam15,sze2013jams,BBJ2021,ChenWang12fano,2016tzacta} for the existence of K\"{a}hler-Einstein metrics with positive scalar curvature. On the other hand, we have acquired some beautiful existence results of Hermitian-Yang-Mills connections in stable holomorphic vector bundles over K\"{a}hler manifolds, see e.g. \cite{1985Anti,UY1986,Donaldson87bundle}. Therefore, this quantization of $\chi(M;E)$ applies to a large class of stable holomorphic vector bundles over K\"{a}hler surfaces.

Finally, by integrated utilization of the following results:
\begin{itemize}
\item[(1)] the classical vanishing theorems \cite{Kodaira1953,mumfordbook66,vanish85book,vanishbook92,holobundle96book},

\item[(2)] the classical Hirzebruch-Riemann-Roch type theorems %(a special case of the Atiyah-Singer index theorem)
    for holomorphic vector bundles over complex manifolds (orbifolds) \cite{Hirzebruchbook1,HRRthmorbifold79,spingeobook89,BlacheChern96},
\item[(3)] the profound convergence theory of K\"{a}hler-Einstein manifolds developed in \cite{Tian1990, tian90icm, CCT2002, tian13c0, DonaldsonSunacta, DonaldsonSunjdg},
\item[(4)] the quantization for the Euler number $\chi(M;E)$ of holomorphic vector bundles over K\"{a}hler-Einstein surfaces (see Theorem \ref{appli33} in Section \ref{apllisec}),
\end{itemize}
we can get the following quantization results:
\begin{thm}\label{intromainapl33}
Let $(M_j, J_j,g_j)$ be a sequence of closed K\"{a}hler-Einstein surfaces with positive (or negative) scalar curvature and with uniformly bounded diameters and Einstein constants. Without loss of generality, we normalize the metrics so that $Ric(g_j)=\pm g_j$.
Let $L_j$ be a holomorphic line bundle over $M_j$, and $A_j$ be a Hermitian-Yang-Mills connection in $L_j$ with uniform $L^2$ Yang-Mills curvature bound. If $(M_j,J_j, g_j, L_j,A_j)$ converges modular bubble connections to some limit
$(M_\infty, J_\infty, g_\infty, L_\infty,A_\infty)$ up to a subsequence, then
\begin{eqnarray*}
\lim_{j\rightarrow \infty}\chi(M_j;L_j) =\chi(M_\infty;L_\infty) -\sum_{l=1}^{u}\mu_{M_{\infty},a_l}(L_\infty) +\sum_{l=1}^{u}\sum_{k=1}^{N_{a_l,\beta}}\bar{\chi}(X_{\beta,k,a_l},A_{\omega^{a_l,\beta,k}}),
\end{eqnarray*}
where $\{a_1,\cdots,a_u\}$ is the set of orbifold singularities in $M_\infty$, $X_{\beta,k,a_l}$ is a Ricci flat K\"{a}hler ALE bubble manifold (orbifold) at the singularity $a_l$, $\bar{\chi}(X,A)$ denotes a quantity related to the bubble connection $A$ in some bundle over the ALE bubble space $X$, and $\mu_{M_{\infty},a_l}(L_\infty)$ is a rational number related to the local structure of the
bundles at the singularity $a_l$.

\begin{itemize}
\item[I.] If $L_j^{*}$ and $L_\infty^{*}$ are ample, then
\begin{eqnarray*}
\lim_{j\rightarrow \infty} dim \ H^{2}(M_j;L_j)= dim\ H^{2}(M_\infty;L_\infty) &-&\sum_{l=1}^{u}\mu_{M_{\infty},a_l}(L_\infty)  \\ &+&\sum_{l=1}^{u}\sum_{k=1}^{N_{a_l,\beta}}\bar{\chi}(X_{\beta,k,a_l},A_{\omega^{a_l,\beta,k}}).
\end{eqnarray*}

\item[II.] If $K_{M_j}^{-1}\otimes L_j$ and $K_{M_{\infty}}^{-1}\otimes L_\infty$ are ample, then
\begin{eqnarray*}
\lim_{j\rightarrow \infty}dim\ H^{0}(M_j;L_j)= dim \ H^{0}(M_\infty;L_\infty) &-& \sum_{l=1}^{u}\mu_{M_{\infty},a_l}(L_\infty)  \\ &+&\sum_{l=1}^{u}\sum_{k=1}^{N_{a_l,\beta}}\bar{\chi}(X_{\beta,k,a_l},A_{\omega^{a_l,\beta,k}});
\end{eqnarray*}
\textbf{or equivalently}, if $L_j$ and $L_\infty$ are ample, and $A_j$ is a sequence of Hermitian-Yang-Mills connections in $K_{M_j}\otimes L_j$ which converge similarly as above, then
\begin{eqnarray}\label{h0111}
\lim_{j\rightarrow \infty}dim\ H^{0}(M_j;K_{M_j}\otimes L_j)&= &dim \ H^{0}(M_\infty;K_{M_\infty}\otimes L_\infty) \\ && -\sum_{l=1}^{u}\mu_{M_{\infty},a_l}(K_{M_\infty}\otimes L_\infty)  \nn \\
&&+\sum_{l=1}^{u}\sum_{k=1}^{N_{a_l,\beta}}\bar{\chi}(X_{\beta,k,a_l},A_{\omega^{a_l,\beta,k}}).\nonumber
\end{eqnarray}
\end{itemize}

\end{thm}

We remark that the quantity $\mu_{M_\infty,a_l}(\cdot)$ in the above theorem is related to the classical Hirzebruch-Riemann-Roch theorem for orbifolds. Let $O$ be a compact $m$-dimensional complex orbifold with isolated singularities $o_1, \cdots, o_s$, and $E$ be a holomorphic vector bundle over $O$. Then by results in \cite{BlacheChern96},
\begin{eqnarray*}
\chi(O;E) = \int_{O}[ch(E)td(TO)]_m +\sum_{k=1}^{s}\mu_{O,o_k}(E),
\end{eqnarray*}
where
 the correction terms $\mu_{O,x_k}(E)$ are rational numbers related to the local structure of the bundles at the singularities, and they are computable in principle as indicated in Section 3 of \cite{BlacheChern96}. The formula in the above is very similar to the Noether formula for
singular surfaces where the Milnor numbers of singularities arise naturally, see Proposition 2.6 of \cite{HPdel}.

In particular, with the help of important and powerful results in \cite{Tian1990,tian13c0,DonaldsonSunacta} etc., which indicate that under the
degeneration for K\"{a}hler-Einstein surfaces with positive (or negative) scalar curvature, for all $p\geq1$ the dimensions
$$dim \ H^{0}(X;K^{-p}_X)$$
(or $dim \ H^{0}(X;K^{p}_X)$, respectively) are invariant, we get the following identities about Euler numbers.
\begin{thm}\label{mainaplintr11}
Let $(M_j, J_j,g_j)$ be a sequence of closed K\"{a}hler-Einstein surfaces with positive (or negative) scalar curvature as in Theorem \ref{intromainapl33}. Suppose that $(M_j, J_j,g_j)$ converges up to a subsequence to a limit space $(M_\infty, J_\infty,g_\infty)$ modular Ricci flat  K\"{a}hler ALE bubble manifolds (orbifolds) $X_1,\cdots,X_v$. Then
\begin{eqnarray*}\label{geoapllpositive}
\sum_{k=1}^{v}\frac{1}{8\pi^2}\int_{X_k}|Rm|^2dvol
=12\sum_{l=1}^{u}\mu_{M_{\infty},a_l}(K_{M_{\infty}}^{-1}),
\end{eqnarray*}
and hence
\begin{eqnarray*}\label{geoapllpositive}
\lim_{j\rightarrow \infty}\chi(M_j)
=\chi_{orb}(M_\infty)
+12\sum_{l=1}^{u}\mu_{M_{\infty},a_l}(K_{M_{\infty}}^{-1})\nonumber
\end{eqnarray*}
when the scalar curvature is positive;
and
 \begin{eqnarray*}\label{geoapllpositive}
\sum_{k=1}^{v}\frac{1}{8\pi^2}\int_{X_k}|Rm|^2dvol
=12\sum_{l=1}^{u}\mu_{M_{\infty},a_l}(K_{M_{\infty}}^{2}),
\end{eqnarray*}
and hence
\begin{eqnarray*}
\lim_{j\rightarrow \infty}\chi(M_j)=\chi_{orb}(M_\infty)
+12\sum_{l=1}^{u}\mu_{M_{\infty},a_l}(K_{M_{\infty}}^{2})
\end{eqnarray*}
when the scalar curvature is negative.
Here $\{a_1,\cdots,a_u\}$ is the set of orbifold singularities of $M_\infty$, $\chi_{orb}(M_\infty)$ is the orbifold Euler number of $M_\infty$, and $\mu_{M_{\infty},a_l}(\cdot)$ is a rational number related to the singularity $a_l$ as above.
\end{thm}

The above theorem says that for large $j$, the Euler number of $M_j$ is determined only by the topology of the limit space $M_\infty$, and we can recover some important geometric information of the ALE bubble spaces from some topological information of the singularities in $M_\infty$. See Example \ref{exam11} and \ref{exam22} for some demonstrations of the applications on two explicit examples.
Next we give some remarks before we explain the theorem.

Firstly, because of the following facts:
\begin{itemize}
\item[(1)] many constraints were found for the possible singularities of the limit orbifolds for the convergence of K\"{a}hler-Einstein surfaces with positive scalar curvature, see e.g. \cite{Tian1990,DonaldsonSunacta,sunjdg16,DonaldsonSunjdg},

\item[(2)] hyper-K\"{a}hler ALE surfaces and more general Ricci flat K\"{a}hler ALE surfaces were well understood, see e.g. \cite{Kronheimer89jdg,Kronheimer89jdg1,SuvainaALE,WrightALE,ChCh1},
\end{itemize}
Theorem \ref{mainaplintr11} should provide a lot of constraints to the behaviour of the degeneration, and hence it should be useful to the study of some aspects of the corresponding moduli spaces. Note that it is proved in \cite{sunjdg16} that the Gromov–Hausdorff compactification agrees with certain algebro-geometric compactification.

Secondly, we expect that Theorem \ref{mainaplintr11} should be helpful to the study of the moduli space of compact K\"{a}hler-Einstein surfaces with negative scalar curvature, on the basis of the recent works \cite{DonaldsonSunacta, DonaldsonSunjdg, liusSze2021cpam} etc.

Now we give some necessary explanations for Theorem \ref{mainaplintr11}.
Firstly, the identities in Theorem \ref{mainaplintr11} are closely related to the classical Riemannian curvature energy identity for non-collapsed Einstein 4-manifolds $(M_j,g_j)$ stated in Theorem \ref{mainconvgethm}. Secondly, recall that \cite{HHorbeuler,Lorbeuler}
\begin{eqnarray}\label{euler11}
\chi_{orb}(M_\infty)=\chi(M_0)+\sum_{k=1}^{u}\frac{1}{|\Gamma_k|}=\chi(M_\infty)-\sum_{k=1}^{u}(1-\frac{1}{|\Gamma_k|}),
\end{eqnarray}
where $|\Gamma_k|$ is the order of the local fundamental group $\Gamma_k$ at $a_k,$  $M_0=M_\infty\setminus \{a_1,\cdots,a_u\}$, and $\chi(M_\infty)$ is the topological Euler number of $M_\infty$. For compact Einstein 4-manifolds (or 4-orbifolds) $M$, it holds that (see e.g. \cite{anderson1989ricci,Jeff2006Curvature})
\begin{eqnarray}\label{euler22}
\chi(M)\,\,(\text{or} \,\,\chi_{orb}(M))=\frac{1}{8\pi^2}\int_M |Rm|^2dvol.
\end{eqnarray}
Thirdly, under the assumptions in Theorem \ref{mainaplintr11},
\begin{eqnarray*}
{\rm vol}(M_j,g_j)\geq c\int_{M_j}c_1(M_j)^2>0
\end{eqnarray*}
for some constant $c=c(\mu)>0$ when the Einstein constant of $g_j$ is uniformly bounded by $\mu>0$, so $(M_j, g_j)$ is non-collapsing, see e.g. \cite{Jeff2006Curvature,DonaldsonSunacta,sunjdg16}.

Next we show some applications of Theorem \ref{mainaplintr11} which may be of some interest to the study of moduli spaces of K\"{a}hler-Einstein Del Pezzo surfaces (with positive scalar curvature). Our arguments rely on two facts. One is that $\mu_{M_{\infty},a_l}(K_{M_{\infty}}^{-1})$ can be computed for non-canonical singularities of type $\frac{1}{r}(b_1,b_2)$ by methods in \cite{ReidYPS,Icecream}, and for canonical (ADE) singularities by using ideas in \cite{ChTs-BMY,LimRota22,LimRota23}, the other is that the classification of possible singularities in the limit orbifolds had been achieved (see e.g. \cite{Tian1990, sunjdg16}).

\begin{thm}\label{mainaplintr22}
Let $(M_j, J_j,g_j)$ be a sequence of closed K\"{a}hler-Einstein surfaces of degree $d$ with positive scalar curvature as in Theorem \ref{intromainapl33}. Suppose that $(M_j, J_j,g_j)$ converges up to a subsequence to a limit space $(M_\infty, J_\infty,g_\infty)$ with ALE bubble spaces as above, then
\begin{eqnarray*}
0<12\sum_{l=1}^{u}\mu_{M_{\infty},a_l}(K_{M_{\infty}}^{-1})<12-d.
\end{eqnarray*}
And the following results hold.
\begin{itemize}
\item[(a)] When $d=3$, the number of singularities of type $A_1$ in $M_\infty$ is $\leq 5$.
\item[(b)] When $d=2,$ the number of singularities of type $A_1$, $A_2$, $A_3$, or $A_4$ in $M_\infty$ is $\leq 6,\, 3,\,  2, \,\,\text{or}\,\,1$ respectively.
 \item[(c)]  When $d=1,$ the number of singularities of type $A_1$, $A_2$, $A_3$, or $A_4$ in $M_\infty$ is $\leq 7,\, 4,\,  2, \,\,\text{or}\,\,2$ respectively; any limit space $M_\infty$ can not have both $A_k$ and $A_l$ singularities if $k+l>9$, in particular, the number of singularities of type $A_5$, $A_6$, $A_7$, or $A_8$ is $\leq 1$; the number of singularities of type $\frac{1}{8}(1,3)$ is $\leq 5.$
 \item[(d)]When $d=1,$ the number of singularities of type $D_4$ is $\leq 2$. And if $M_\infty$ has one $D_4$ singularity, then it has no $A_k$ singularities with $k>5$; if $M_\infty$ has two $D_4$ singularities, then it has at most a $\frac{1}{4}(1,1)$ singularity or a $\frac{1}{9}(1,2)$ singularity besides these two singularities.
\end{itemize}

\end{thm}

It had been shown in \cite{sunjdg16} that when $d=3,4$ the singularities must be canonical. Actually for $d = 4$, $Sing(M_\infty)$ consists of only two or four $A_1$ singularities; for $d = 3$, $Sing(M_\infty)$ consists of only points of type $A_1$, or of exactly
three points of type $A_2$. And it is well know that for $d>4$ the moduli space is
just a single point.

\begin{rem}
The bounds in the theorem are quite effective. In fact (see Page 165 of \cite{sunjdg16}), if we assume that $Sing(M_\infty)$ consists of only canonical singularities, then for $d = 2$, $Sing(M_\infty)$ consists of only points of type $A_1, A_2$, or of exactly
two $A_3$ singularities; for $d = 1$, $Sing(M_\infty)$ consists of only points of type $A_k$ $(k \leq 7)$, or of
exactly two $D_4$ singularities. And there exist limit spaces with one $A_8$ singularity and two
$\frac{1}{9}(1, 2)$ singularities (Example 3.10 of \cite{sunjdg16}), and limit spaces with two $D_4$ singularities and one $\frac{1}{4}(1,1)$ singularity (Example 5.19 of \cite{sunjdg16}).

 We think that if we trace the process of constructing moduli spaces in \cite{sunjdg16} carefully, we could be able to obtain more precise information about singularities. However the arguments in this paper should be valuable in some sense since we can make some preliminary judgments without resorting to subtle analysis of moduli spaces. Indeed the identities in Theorem \ref{mainaplintr11} are complementary to Noether formula  for
singular surfaces in Proposition 2.6 of \cite{HPdel} which had played a significant role in the arguments of \cite{sunjdg16}, and together they can lead to very effective constraints.
\end{rem}

 In fact, we have following identity which connects Milnor numbers for singularities and the correction terms $\mu_{M_{\infty},p}(K_{M_{\infty}}^{-1})$ in the Hirzebruch-Riemann-Roch theorem for orbifolds.
\begin{prop}\label{HRR-Milnor}
For each $d=1,\cdots,4$, let $M_\infty $ be a limit orbifold as in Theorem \ref{mainaplintr22}, then we have
\begin{eqnarray*}
\sum_{p\in Sing(M_\infty)}(1-\frac{1}{n_p})+\sum_{p\in Sing(M_\infty)}\nu_p=12\sum_{p\in Sing(M_\infty)}\mu_{M_{\infty},p}(K_{M_{\infty}}^{-1}),
\end{eqnarray*}
and it follows that
\begin{eqnarray*}
\rho(M_\infty)+12\sum_{p\in Sing(M_\infty)}\mu_{M_{\infty},p}(K_{M_{\infty}}^{-1})-\sum_{p\in Sing(M_\infty)}(1-\frac{1}{n_p})=10-d,
\end{eqnarray*}
where $\nu_p$ is the Milnor number for $p$, $n_p$ is the order of the local fundamental group at $p$, and
$\rho(M_\infty)$ is the Picard rank of $M_\infty$.
\end{prop}

\begin{rem}
In \cite{sunjdg16}(see Remark 5.14 therein) the authors gave explicit examples of limit orbifolds $M_\infty$ which have exactly two $\frac{1}{4}(1,1)$ singularities in the degree 2 case. By \cite{ReidYPS,Icecream}, for each $\frac{1}{4}(1,1)$ singularity, $\mu_{M_{\infty},a_l}(K_{M_{\infty}}^{-1})=\frac{1}{16},$ so by Theorem \ref{mainaplintr11}
\begin{eqnarray*}
\sum_{k=1}^{v}\frac{1}{8\pi^2}\int_{X_k}|Rm|^2dvol
=12\sum_{l=1}^{2}\mu_{M_{\infty},a_l}(K_{M_{\infty}}^{-1})=\frac{3}{2}.
\end{eqnarray*}
It is obvious that $v\geq 2$, and by the results in \cite{SuvainaALE,WrightALE} the Ricci flat K\"{a}hler ALE bubble spaces are given by quotients of hyper-K\"{a}hler ALE surfaces in this case, and the minimum value for $$E(X_k)=\int_{X_k}|Rm|^2dvol$$
is $6\pi^2$.
Therefore if $M_j$ converge to $M_\infty$, then exactly two ALE bubble spaces with curvature energy $E(X_k)=6\pi^2$ will appear. And in Subsection \ref{moduliappli} we will show the effectiveness of Theorem \ref{mainaplintr11} in some explicit examples given in \cite{sunjdg16} by computing the orbifold Euler numbers of the orbifolds in the examples directly. For instance, we can get that $\chi_{orb}(M_\infty)=\frac{17}{2}$ for these orbifolds by methods of algebraic topology(see Example \ref{exam11}), thus it is consistent with the result obtained by using Theorem \ref{mainaplintr11} by noticing that $\chi(M_j)=10$ in this case.
\end{rem}

\begin{rem}
For the negative scalar curvature case, let $M_j$ be a sequence of closed K\"{a}hler-Einstein surfaces with negative scalar curvature and Euler number $N_e$ which converge to $M_\infty$ as in Theorem \ref{mainaplintr11}, and if $M_\infty$ has only canonical singularities, then we can give some preliminary bounds on the number of singularities as above. Indeed, it holds that
\begin{eqnarray*}
0<12\sum_{l=1}^{u}\mu_{M_{\infty},a_l}(K_{M_{\infty}}^{2})<N_e,
\end{eqnarray*}
and
\begin{eqnarray*}
12\mu_{M_{\infty},a_l}(K_{M_{\infty}}^{2})=(n+1-\frac{1}{n+1}),\, (n+1-\frac{1}{4(n-2)}),\, (7-\frac{1}{24}),\,(8-\frac{1}{48}),\,(9-\frac{1}{120})
\end{eqnarray*}
for singularities of type $A_n,$ $D_n,$ and $E_6$, $E_7$,and $E_8$ by
\cite{ChTs-BMY,LimRota22,LimRota23}. We have to acknowledge that such bounds may be rather crude and requires additional theoretical support to determine the possible types of singularities.

The quantity $\mu_{M_{\infty},a_l}(L_\infty) $ in Theorem \ref{intromainapl33} is computable for more general line bundles by the methods in \cite{ReidYPS,Icecream} and \cite{ChTs-BMY,LimRota22,LimRota23}. So we may get some interesting results by applying it to other line bundles which are important in various cases.

\end{rem}

Recently, it was shown that one can apply Riemann-Roch polynomials which are equal to the Euler numbers of some holomorphic line bundles to the classification of hyper-K\"{a}hler fourfold (complex dimension) \cite{2022Computing}. It may be hopeful to combine the ideas in \cite{2022Computing} and the results about quantization of Euler numbers of holomorphic vector bundles over a sequence of higher dimensional K\"{a}hler-Einstein manifolds analogous to
Theorem \ref{appli33} to get some further applications.

\begin{rem}\label{rmkdegeymenergyhd}
In higher dimensional cases, if we replace the $L^2$ bounds for Riemannian curvatures and Yang-Mills curvatures (see \eqref{cheegernaberjiang} and \eqref{ymbounds}) by the corresponding $L^{\frac{n}{2}}$ bounds, we can obtain similar results for Yang-Mills fields over (real)  $n$-dimensional non-collapsing Einstein manifolds as Theorem \ref{degeymenergyintr} and Theorem \ref{appli33} by applying similar arguments as in Section \ref{treeidentity}.
See Theorem \ref{degeymenergyhd} and Remark \ref{hgdimrem} for the corresponding quantization results in higher dimensional cases.

Theorem \ref{intromainapl33} can be extended to the (complex) $m$-dimensional case for any integer $m>2$ by imposing some further $L^{m}$ bound assumptions on Riemannian curvature and Yang-Mills curvature, see Theorem \ref{highdimthm}. See also Theorem \ref{volquant} for the quantization of the volume of  line bundles over K\"{a}hler-Einstein manifolds.
\end{rem}

Inspired by the substantial and profound works e.g. \cite{cheegercoldingJdg1,cheegercoldingJdg2,cheegercoldingJdg3,cheeger2015regularity,jiangnaber21,RiviequantYM,Naberyangenergy,tian2000y,taotian,Uhlenbkato,LXZ-Kstab,Li-KStabEins}, we think it is of great interest to investigate the qualitative behavior for a sequence of Yang-Mills connections $A_j$ with uniform $L^2$ bounds for Yang-Mills curvatures in vector bundles $E_j$ over non-collapsed degenerating closed Einstein $n$-manifolds $(M_j, g_ j)$ $(n>4)$ with uniformly bounded Einstein constants and bounded diameters.

Now we turn to the main ideas for the proof of Theorem \ref{degeymenergyintr}.
It is well known that Yang-Mills fields can blow-up at points where Yang-Mills energy concentrate and the Einstein metrics of the underlying 4-manifolds can also blow-up at points where the curvature energy concentrate. To study the compactness problem, the first difficulty is that the bubble-neck decomposition for Yang-Mills fields over 4-manifolds with varying Einstein metrics is somewhat subtle, since the blow-up of Yang-Mills fields and the blow-up of the underlying domain manifolds are entangled with each other. The second difficulty is that the degeneration of non-collapsed Einstein 4-manifolds is accompanied by the blow-up of Riemannian curvatures, therefore we need to pay attention to the effects of the Riemannian curvatures to the asymptotic analysis for Yang-Mills fields over degenerating neck regions.

For the proof of Theorem \ref{degeymenergyintr}, the crucial step is to prove that there is no Yang-Mills energy loss on the degenerating neck regions. To achieve this, we shall explore the fine structure of Yang-Mills equations and establish two key differential inequalities for Yang-Mills connections in vector bundles over Riemannian manifolds (see Lemma \ref{keyymequ}), and then derive decay estimates for Yang-Mills fields over degenerating neck regions. It is worth mentioning that the techniques developed in \cite{Naberyangenergy} should also give a (more involved) proof of Theorem \ref{degeymenergyintr} by applying some techniques in the degenerating manifolds situations.

Finally, we would like to remark that Yang-Mills fields over manifolds with metrics degenerating in different ways have been studied in e.g. \cite{Chen1998Con,ChenComplex, FukayaAsd98, ASDK32021}.

\textbf{Notation.} In this paper, $C$ represents a universal constant which may vary from line to line. $K_X^{2}$ denotes $K_X\otimes K_X$ and so on. And $n$ represents the real dimension of a manifold, and $m$ is the complex dimension if we do not explicitly mention in the following.

The paper is organized as follows. In Section \ref{apllisec}, we give some applications of Theorem \ref{degeymenergyintr}. In particular, we will give the proof of Theorem \ref{mainaplintr11} in Subsection \ref{mainappliprf}, and prove Theorem \ref{mainaplintr22} and Proposition \ref{HRR-Milnor} in Subsection \ref{moduliappli}. In Section \ref{Preliminaries111}, we recall some background knowledge about the bubble tree convergence of non-collapsed degenerating Einstein 4-manifolds and present some key results for the blow-up analysis of Yang-Mills fields. In Section \ref{treeidentity}, we do the bubble-neck decomposition for the blow-up analysis of Yang-Mills fields over non-collapsed degenerating Einstein 4-manifolds, prove Theorem \ref{degeymenergyintr} and state a higher dimensional analog of Theorem \ref{degeymenergyintr}.

\

\section{Geometric and topological applications}\label{apllisec}

\vskip5pt

In this section, we present some geometric and topological applications of the energy quantization result (Theorem \ref{degeymenergyintr}) for Yang-Mills fields over non-collapsed degenerating Einstein 4-manifolds. The arguments used in this section are in fact independent on the detailed proof of Theorem \ref{degeymenergyintr}.

\subsection{Quantization of the first Pontrjagin number}

Recall that
\begin{eqnarray*}
-\text{tr}_\mathfrak{g}(F_{A}\wedge F_{A})=|F_{A}^{+}|^2-|F_{A}^{-}|^2,
\end{eqnarray*}
where $F_{A}^{-}$ is the anti-self-dual part of $F_{A}$ (i.e., $\ast F_{A}^{-}=-F_{A}^{-}$), and $F_{A}^{+}$ is the self-dual part of $F_{A}$, $\text{tr}_\mathfrak{g}$ is the trace on $\mathfrak{g}$ (which is the Lie algebra of $G$) given by the Killing form. A direct consequence of Theorem \ref{degeymenergyintr} is the following quantization of the first Pontrjagin number.
\begin{thm}\label{degeymenergycor}
Let $M_j, g_j, E_j,A_j$ be as in Theorem \ref{degeymenergyintr}, then we have
\begin{eqnarray*}
\lim_{j\rightarrow \infty}\int_{M_j}{\rm tr}_\mathfrak{g}(F_{A_j}\wedge F_{A_j})& =&\int_{M_\infty}{\rm tr}_\mathfrak{g}(F_{A_\infty}\wedge F_{A_\infty}) +\sum_{l=1}^{v}\sum_{k=1}^{N_{b_l}}\mathcal{YMP}(\mathbb{R}^4, A_{\omega_{b_l,k}})\\
&& +\sum_{l=1}^{u}\sum_{\eta=\alpha,\beta,\gamma}\sum_{k=1}^{N_{a_l,\eta}} \mathcal{YMP}(X_{\eta,k,a_l}, A_{\omega^{a_l,\eta,k}}),
\end{eqnarray*}
where $\mathcal{YMP}(X, A)$ is the integral of the corresponding 4-form ${\rm tr}_\mathfrak{g}(F_{A}\wedge F_{A})$ for the connection $A$ over the underlying space $X$. Here $X_{\alpha,k,a_l}$ is $\mathbb{R}^4$, $X_{\beta,k,a_l}$ is a Ricci flat ALE space and $X_{\gamma,k,a_l}$ is of the form $\mathbb{R}^4/\Gamma$ for some nontrivial finite group $\Gamma \subset SO(4)$.
\end{thm}

By Chern-Weil theory, for any connection $A$ in a bundle $E$ with structure group $G$ over a closed 4-manifold $M$ (see e.g. \cite{Taubes1982Self}),
\begin{eqnarray*}
p_1(E)= -\frac{1}{4\pi^2}\int_{M}\text{tr}_\mathfrak{g}(F_{A}\wedge F_{A}).
\end{eqnarray*}
Therefore, Theorem \ref{degeymenergycor} says that there exists a certain kind of quantization for the first Pontrjagin number $p_1$, provided the existence of Yang-Mills connections.

\subsection{Quantization of the Euler number of holomorphic vector bundles}

Recall that the Euler number of a holomorphic vector bundle $E$ over a closed complex manifold $M$ is
\begin{eqnarray*}
\chi(M;E)=\sum_{q=0}^{m}(-1)^q \ {\rm dim} H^ q (M;E), \quad m={\rm dim}_{\mathbb{C}}M.
\end{eqnarray*}
The classical Hirzebruch-Riemann-Roch theorem e.g. \cite{Hirzebruchbook1,spingeobook89} states that
\begin{eqnarray}\label{HRR111}
\chi(M;E)=\int_{M}[ch(E) \ td(TM)]_m,
\end{eqnarray}
where $ch(E)$ is the Chern character, $td(TM)$ is the Todd polynomial, and $[\quad ]_m $ means the component of degree $m.$

It is easy to see that, for any metric $g$ on a closed 2-dimensional complex manifold $M$ and for any connection $A$ in $E$ over $M$, we have that for any compact domain $K\subset M$
\begin{eqnarray*}
\int_{K}[ch(E) \ td(TM)]_2\leq C\int_{K}\left(|Rm(g)|^2+|F_A|^2\right)dV_g
\end{eqnarray*}
for some universal constant $C>0$. Therefore, similarly to Theorem \ref{degeymenergycor}, we can obtain a quantization for $\chi(M;E)$  by applying Theorem \ref{degeymenergyintr} and the classical Riemannian curvature energy identity as stated in Theorem \ref{mainconvgethm} in Section \ref{Preliminaries111}. More precisely, we have

\begin{thm}\label{appli33}
Let $M_j, g_j, E_j,A_j$ be as in Theorem \ref{degeymenergyintr}, and furthermore, suppose that $(M_j, J_j, g_j)$ is a K\"{a}hler-Einstein surface, $E_j$ is a holomorphic vector bundle over $M_j$ and $A_j$ is a Yang-Mills connection in $E_j$. Then we have
\begin{eqnarray*}
\lim_{j\rightarrow \infty}\int_{M_j} [ch(E_j) \ td(TM_j)]_2 &=&\int_{M_\infty}[ch(E_\infty) \ td(TM_\infty)]_2 \\
& &+\sum_{l=1}^{v}\sum_{k=1}^{N_{b_l}}\bar{\chi}(\mathbb{C}^2,A_{\omega_{b_l,k}}) \nn \\
&& +\sum_{l=1}^{u}\sum_{\eta=\alpha,\beta,\gamma}\sum_{k=1}^{N_{a_l,\eta}}\bar{\chi}(X_{\eta,k,a_l},A_{\omega^{a_l,\eta,k}}),
\end{eqnarray*}
where $\bar{\chi}(X,A)$ is the integral of the corresponding 4-form $[ch(E) \ td(TX)]_2$ for the connection $A$ over the underlying space $X$. Here $X_{\alpha,k,a_l}$ is $\mathbb{C}^2$, $X_{\beta,k,a_l}$ is a Ricci flat ALE space and $X_{\gamma,k,a_l}$ is of the form $\mathbb{C}^2/\Gamma$ for some nontrivial finite group $\Gamma \subset U(2)$.
\end{thm}

\begin{rem}\label{keyrem1111}
When the underlying space $X$ is not a compact manifold, in general, the integral $\bar{\chi}(X,A)$ of the 4-form $[ch(E)td(TX)]_2$ in Theorem \ref{appli33} may depend on the metric on $X$ and the connection in $E$. For a connection $A$ in $E$ over $(X,h)$, it equals to
\begin{eqnarray*}
\int_{X}\frac{r}{12}\left((c_1(X,h))^2+c_2(X,h)\right)+\frac{1}{2}c_1(X,h)c_1(E,A)+\frac{1}{2}(c_1(E,A))^2-c_2(E,A),
\end{eqnarray*}
where $r$ is rank of the bundle $E$, $c_k(X,h)$ and $c_k(E,A)$ are the standard integrands defining the Chern classes
in Chern-Weil theory, and
\begin{eqnarray*}
c_2(E,A)-\frac{1}{2}(c_1(E,A))^2=\frac{1}{8\pi^2}tr(F_A\wedge F_A).
 \end{eqnarray*}
 Here $tr$ is just the usual trace operator for complex matrix, for example, on the Lie
algebra $\mathfrak{u}(n)$ of skew adjoint matrices, $tr(\xi^2 ) = -|\xi|^2.$
\end{rem}

Firstly, notice that the Hirzebruch-Riemann-Roch theorem was proved to hold for some singular spaces, see e.g. \cite{HRRthmorbifold79,BlacheChern96}. In particular, it holds for complex orbifolds, see \cite{BlacheChern96}. Let $O$ be a compact $m$-dimensional complex orbifold with isolated orbifold singularities $x_1, \cdots, x_s$, and $E$ be a holomorphic vector bundle over $O$. Then
\begin{eqnarray}\label{eulerformula222}
\chi(O;E) = \chi_{orb}(O;E) +\sum_{k=1}^{s}\mu_{O,x_k}(E),
\end{eqnarray}
where
\begin{eqnarray*}
\chi_{orb}(O;E)=\int_{O}[ch(E)td(TO)]_m
\end{eqnarray*}
and the correction terms $\mu_{O,x_k}(E)$ are rational numbers related to the local fundamental groups. For the precise statement and proof, we refer to Section 3 of \cite{BlacheChern96}. Here a holomorphic bundle on the complex orbifold $O$ is a holomorphic vector bundle $E$ on the
regular part of $O$ such that for each local uniformization  $\pi_{o_k}:\widetilde{U_{o_k}}\rightarrow O$
of a singular point $o_k$, the pull-back $\pi^*E$  on $\widetilde{U_{o_k}}\setminus \pi^{-1}(o_k) $ can be extended to the whole $\widetilde{U_{o_k}}$, see e.g. \cite{Tian1990}. The notion of vector bundles over Riemannian orbifolds is defined in a similar way.

Secondly, when $(X,h)$ is a Ricci flat ALE K\"{a}hler manifold (orbifold) and $A$ is a connection in $E$ over $X$ as in Theorem \ref{appli33}, note that $c_1(X,h)=0$, then we have (see Remark \ref{keyrem1111})
\begin{eqnarray}\label{eulerformula333}
\bar{\chi}(X,A)&=&\int_{X}\frac{r}{12}c_2(X,h)+\frac{1}{2}(c_1(E,A))^2-c_2(E,A)\\
&=&\int_{X}\frac{r}{12} \cdot \frac{1}{8\pi^2}|Rm(h)|^2dV_h-\frac{1}{8\pi^2}\int_{X}tr(F_A\wedge F_A).\nonumber
\end{eqnarray}
Here we have used the fact that
\begin{eqnarray*}
\frac{1}{8\pi^2}|Rm(h)|^2dV_h= \left[a_m(c_1(X,h))^2+b_mc_2(X,h)\right]\wedge (\omega_h)^{n-2}
\end{eqnarray*}
holds pointwise for K\"{a}hler-Einstein metrics, where $\omega_h$ is the K\"{a}hler form corresponding to $h$, $a_m$ and $b_m$ are constants depending only on the complex dimension $m$, see e.g. \cite{2013Volume}. In particular, $a_2=0$, $b_2=1$.

Thirdly, in the case that $M$ is $\mathbb{C}^2$, or $\mathbb{C}^2/\Gamma$ (equipped with the flat metric), let $A$ be a connection in some bundle $E$ over $M$, then
\begin{eqnarray}\label{eulerformula4444}
\bar{\chi}(M,A)&=&\int_{M}\frac{1}{2}(c_1(E,A))^2-c_2(E,A)\\&=& -\frac{1}{8\pi^2} \int_{M}tr(F_A\wedge F_A).\nonumber
\end{eqnarray}

\subsection{Quantization results for holomorphic line bundles}

Inspired and enlightened by the beautiful works about line bundles over K\"{a}hler manifolds by Tian, Donaldson-Sun etc, we focus on holomorphic line bundles in this subsection.

Let $(M_j, J_j, g_j)$ be a sequence of K\"{a}hler-Einstein surfaces as in Theorem \ref{appli33}, and let $L_j$ be a holomorphic line bundle over $M_j$. $L_j$ is stable, since its rank is 1, see e.g. \cite{holobundle96book}. Therefore there exists a Hermitian-Yang-Mills
connection $A_j$ in $L_j$, see e.g. \cite{1985Anti,UY1986,Donaldson87bundle}. Suppose $(M_j,J_j, g_j, L_j,A_j)$ converges to $(M_\infty, J_\infty,g_\infty, L_\infty,A_\infty)$ (modular bubbles) up to a subsequence, then we have the following quantization result for holomorphic line bundles.

\begin{thm}\label{appli44}
Let $(M_j,J_j, g_j, L_j,A_j)$ and all the other notations be as above. Then
\begin{eqnarray*}
\lim_{j\rightarrow \infty}\int_{M_j} [ch(L_j)td(TM_j)]_2 &=&\int_{M_\infty}[ch(L_\infty)td(TM_\infty)]_2 \\
& &+\sum_{l=1}^{v}\sum_{k=1}^{N_{b_l}}\bar{\chi}(\mathbb{C}^2,A_{\omega_{b_l,k}}) \\
& & +\sum_{l=1}^{u}\sum_{\eta=\alpha,\beta,\gamma}\sum_{k=1}^{N_{a_l,\eta}}\bar{\chi}(X_{\eta,k,a_l},A_{\omega^{a_l,\eta,k}}),
\end{eqnarray*}
where $\bar{\chi}(X,A)$ is the integral of the corresponding 4-form $[ch(L)td(TX)]_2$ for the connection $A$ in $L$ over the underlying space $X$ as before.
\end{thm}

Firstly, by the arguments in previous subsections, in particular \eqref{eulerformula222}, we have
\begin{eqnarray}\label{eulerformula22211}
\chi(M_\infty;L_\infty) = \int_{M_\infty}[ch(L_\infty)td(TM_\infty)]_2 +\sum_{l=1}^{u}\mu_{M_{\infty},a_l}(L_\infty),
\end{eqnarray}
if $M_\infty$ has orbifold singularities $a_1,\cdots,a_u$.

Secondly, when $(X,h)$ is a Ricci flat ALE K\"{a}hler manifold (orbifold) and $A_{\omega^\beta}$ is a bubble connection in some holomorphic line bundle $L$ over $X$, noticing that the Ricci curvature of $X$ is zero and the rank of $L$ is 1, then  by applying (\ref{eulerformula333}) we have that
\begin{eqnarray}\label{importformula111}
\bar{\chi}(X,A_{\omega^\beta})&=&\int_{X}\frac{1}{12}c_2(X,h)+\frac{1}{2}(c_1(L,A_{\omega^\beta}))^2\nonumber\\
&=&\int_{X}\frac{1}{12}\cdot\frac{1}{8\pi^2}|Rm(h)|^2dV_h-\frac{1}{8\pi^2}\int_{X}(F_{A_{\omega^\beta}}\wedge F_{A_{\omega^\beta}}).
\end{eqnarray}

Thirdly, for a bubble connections $A$ in a holomorphic line bundle $L$ on $\mathbb{C}^2$ or $\mathbb{C}^2/\Gamma$, we have by (\ref{eulerformula4444}) that
\begin{eqnarray*}
\bar{\chi}(M,A)=-\frac{1}{8\pi^2}\int_{M}(F_{A}\wedge F_{A}),
\end{eqnarray*}
for the cases that $M$ is $\mathbb{C}^2$ or $\mathbb{C}^2/\Gamma$.
By using the removable singularity Theorem \ref{remsingu} in Section \ref{Preliminaries111} and taking a lifting if necessary, we can get from $A$ a smooth connection $\bar{A}$ in some complex line bundle over $S^4$. Because $H^2(S^{4},\mathbb{Z})=0,$ all complex line bundles on $S^4$ are trivial e.g. \cite{holobundle96book}, it follows that
\begin{eqnarray*}
\int_{S^4}(F_{\bar{A}}\wedge F_{\bar{A}})=0,
\end{eqnarray*}
and hence
\begin{eqnarray*}
\bar{\chi}(M,A)=0
\end{eqnarray*}
when $M$ is $\mathbb{C}^2$ or $\mathbb{C}^2/\Gamma$.

Thus we have the following identity by combining the Hirzebruch-Riemann-Roch theorems stated above (see (\ref{HRR111}) and (\ref{eulerformula22211})) and Theorem \ref{appli44}. Note that $\bar{\chi}(M,A)=0$ for connections with finite Yang-Mills energy in holomorphic line bundles over $M$ when $M$ is $\mathbb{C}^4$ or $\mathbb{C}^4/\Gamma$.

\begin{thm}\label{appli55}
Let $(M_j,J_j, g_j, L_j,A_j)$ and all the other notations as above. Then we have
\begin{eqnarray*}
\lim_{j\rightarrow \infty}\chi(M_j;L_j) =\chi(M_\infty;L_\infty) -\sum_{l=1}^{u}\mu_{M_{\infty},a_l}(L_\infty) +\sum_{l=1}^{u}\sum_{k=1}^{N_{a_l,\beta}}\bar{\chi}(X_{\beta,k,a_l},A_{\omega^{a_l,\beta,k}}),
\end{eqnarray*}
 where $\bar{\chi}(X,A)$ and $\mu_{M_\infty,a_l}(L_\infty)$ are defined as before.
\end{thm}

Next, we recall a classical vanishing theorem which will be used later. For more about vanishing theorems, we refer to e.g. \cite{Kodaira1953,mumfordbook66,vanish85book,vanishbook92,holobundle96book}.

\begin{thm}\label{vanishingthm00}(\cite[Chapter 7]{vanish85book})
Let $L$ be an ample holomorphic line bundle on an irreducible, $m$-dimensional, normal projective-algebraic variety X. If $X$ is nonsingular or X has only rational singularities, then
\begin{eqnarray*}
H^q(X;L^*)=0, \quad {\it for}\  q<m,
\end{eqnarray*}
and
\begin{eqnarray*}
H^q(X;K_X\otimes L)=0, \quad {\it for}\  q>0.
\end{eqnarray*}

\end{thm}

Now we state some important theorems about the convergence of K\"{a}hler-Einstein manifolds \cite{Tian1990, tian13c0, DonaldsonSunacta, DonaldsonSunjdg}.
\begin{thm}\label{tianconvegthm} (\cite{anderson1989ricci,BKN,Tian1990,sunjdg16} etc.)
Given a sequence of  K\"{a}hler-Einstein Del Pezzo surfaces $(X_ i ,J _i,g_i )$ of degree $d$, then by passing to a subsequence if necessary, it converges in the Gromov-Hausdorff sense to an irreducible K\"{a}hler-Einstein log Del Pezzo surface $(X _\infty ,J_\infty,g_\infty)$ of degree $d$.
\end{thm}

In the above, a log Del Pezzo surface is a normal projective surface with quotient singularities and ample anti-canonical divisor.

\begin{thm} \label{Donldsonsuncoverge} (\cite{DonaldsonSunacta})
 Given a sequence of $m$-dimensional K\"{a}hler-Einstein manifolds $(X_ i, J _i, g_i)$ with Ricci curvature 1 (resp. $-1$), with bounded diameters (or volumes) and satisfying the volume non-collapsing property, by passing to a subsequence, it
converges in the Gromov-Hausdorff sense to a projective variety with log-terminal singularities $(X _\infty ,J_\infty )$
endowed with a weak K\"{a}hler-Einstein metric $g_\infty$ (e.g. \cite{EGZ09jams}). Moreover,
there exist integers $k$ and $N$, depending only on $m$, so that we could embed $X_i \ (i \in \mathbb{N}\bigcup\{\infty\})$ into $\mathbb{CP}^N$ using the orthonormal basis of
$H^0(X_i; K_{X_i}^{-k})$ (resp. $H^0(X_i; K_{X_i}^{k})$) such that $X_i$ converges to $X_\infty$ as varieties in $\mathbb{CP}^N$.

\end{thm}
For more related convergence results of K\"{a}hler-Einstein metrics see e.g. \cite{tian13c0,DonaldsonSunacta,DonaldsonSunjdg,liusSze2021cpam}.

\begin{rem}\label{reminportant}
When the dimension $m$ is 2, since a two-dimensional log terminal singularity is
a quotient singularity \cite{Kawamata1988}, all singularities are (isolated) quotient singularities. Hence they are rational singularities, see e.g. Page 50 of \cite{vanishbook92}. So in the 2-dimensional case the limit space in Theorem \ref{Donldsonsuncoverge} is an irreducible normal projective variety with at worst isolated rational singularities.

\end{rem}

By combining the above results, we can prove Theorem \ref{intromainapl33}.

\begin{proof}[\textbf{Proof of Theorem \ref{intromainapl33}}]
Notice that Theorem \ref{tianconvegthm}, Theorem \ref{Donldsonsuncoverge} and Remark \ref{reminportant} indicate that $M_\infty$ is a normal projective variety with at most rational singularities. And $M_j$ is a projective-algebraic manifold for every $j$ thanks to the Kodaira embedding theorem and a theorem of Chow, see e.g. Chapter 5 of the book by Huybrechts \cite{Huybrechts05book}. It follows that we can apply Theorem \ref{vanishingthm00} to the corresponding line bundles $L_j$ over $M_j$ and $L_\infty$ over $M_\infty$. Recall that by the definition we have that for a holomorphic line bundle $L$ over a $m$-dimensional normal projective-algebraic variety $X$,
\begin{eqnarray*}
\chi(X;L)=\sum_{k=0}^{m}(-1)^k dim H^{k}(X;L).
\end{eqnarray*}
So by the vanishing Theorem \ref{vanishingthm00},
 if $L_j^{*}$  and $L_{\infty}^{*}$ are ample, then
 \begin{eqnarray*}
\chi(M_j;L_j)=dim H^{2}(M_j;L_j) \text{ and } \chi(M_\infty;L_\infty)=dim H^{2}(M_\infty;L_\infty);
\end{eqnarray*}
and if $L_j$ and $L_\infty$ are ample, then
$$\chi(M_j;K_{M_j}\otimes L_j)=dim H^{0}(M_j;K_{M_j}\otimes L_j),$$
and
$$\chi(M_\infty;K_{M_\infty}\otimes L_\infty)=dim H^{0}(M_\infty;K_{M_\infty}\otimes L_\infty).$$
Therefore, Theorem \ref{intromainapl33} follows immediately from Theorem \ref{appli55}.
\end{proof}

Now we extend Theorem \ref{intromainapl33} to higher dimensional cases.

\begin{thm}\label{highdimthm}
Let $(M_j, J_j,g_j)$ be a sequence of (complex) $m$-dimensional closed K\"{a}hler-Einstein manifolds with positive (or negative) scalar curvature as in Theorem \ref{Donldsonsuncoverge}, and moreover assume that they have uniform $L^{m}$ bound on Riemannian curvatures.
Let $L_j$ be a holomorphic line bundle over $M_j$, and $A_j$ be a Hermitian-Yang-Mills connection in $L_j$ with uniform $L^m$ Yang-Mills curvature bound. If $(M_j,J_j, g_j, L_j,A_j)$ converges modular bubble connections to some limit
$(M_\infty, J_\infty, g_\infty, L_\infty,A_\infty)$ up to a subsequence, then
\begin{eqnarray*}
\lim_{j\rightarrow \infty}\chi(M_j;L_j) =\chi(M_\infty;L_\infty) -\sum_{l=1}^{u}\mu_{M_{\infty},a_l}(L_\infty) +\sum_{l=1}^{u}\sum_{k=1}^{N_{a_l,\beta}}\bar{\chi}(X_{\beta,k,a_l},A_{\omega^{a_l,\beta,k}}),
\end{eqnarray*}
where $\{a_1,\cdots,a_u\}$ is the set of orbifold singularities in $M_\infty$, $X_{\beta,k,a_l}$ is a Ricci flat ALE bubble manifold (orbifold) at the singularity $a_l$, $\bar{\chi}(X,A)$ is the integral of the corresponding $2m$-form $[ch(L)td(TX)]_m$ for the connection $A$ in $L$ over the underlying space $X$ as before, and $\mu_{M_{\infty},a_l}(L_\infty)$ is a rational number related to the local fundamental group at $a_l$ as above.

\begin{itemize}
\item[I.] If $L_j^{*}$ and $L_\infty^{*}$ are ample, then
\begin{eqnarray*}
\lim_{j\rightarrow \infty} dim \ H^{m}(M_j;L_j)= dim\ H^{m}(M_\infty;L_\infty) &-&\sum_{l=1}^{u}\mu_{M_{\infty},a_l}(L_\infty)  \\ &+&\sum_{l=1}^{u}\sum_{k=1}^{N_{a_l,\beta}}\bar{\chi}(X_{\beta,k,a_l},A_{\omega^{a_l,\beta,k}}).
\end{eqnarray*}

\item[II.] If $K_{M_j}^{-1}\otimes L_j$ and $K_{M_{\infty}}^{-1}\otimes L_\infty$ are ample, then
\begin{eqnarray*}
\lim_{j\rightarrow \infty}dim\ H^{0}(M_j;L_j)= dim \ H^{0}(M_\infty;L_\infty) &-& \sum_{l=1}^{u}\mu_{M_{\infty},a_l}(L_\infty)  \\ &+&\sum_{l=1}^{u}\sum_{k=1}^{N_{a_l,\beta}}\bar{\chi}(X_{\beta,k,a_l},A_{\omega^{a_l,\beta,k}});
\end{eqnarray*}
\textbf{or equivalently}, if $L_j$ and $L_\infty$ are ample, and $A_j$ is a sequence of Hermitian-Yang-Mills connections in $K_{M_j}\otimes L_j$ which converge similarly as above, then
\begin{eqnarray*}
\lim_{j\rightarrow \infty}dim\ H^{0}(M_j;K_{M_j}\otimes L_j)&= &dim \ H^{0}(M_\infty;K_{M_\infty}\otimes L_\infty) \\ && -\sum_{l=1}^{u}\mu_{M_{\infty},a_l}(K_{M_\infty}\otimes L_\infty)  \nn \\
&&+\sum_{l=1}^{u}\sum_{k=1}^{N_{a_l,\beta}}\bar{\chi}(X_{\beta,k,a_l},A_{\omega^{a_l,\beta,k}}).\nonumber
\end{eqnarray*}
\end{itemize}

\end{thm}

\begin{rem}
Its proof is similar to the one for Theorem \ref{intromainapl33}. One can see Remark \ref{rmkdegeymenergyhd} for the quantization of some key quantities in higher dimensions.
Notice that all the singularities of $M_\infty$ locally look like $\mathbb{C}^m/\Gamma$ for a finite group $\Gamma\subset U(m)$ and hence are isolated quotient singularities under the conditions of Theorem \ref{highdimthm}. $\Gamma\subset SO(2m)$ is well known, see e.g. \cite{anderson1989ricci}, and since $M_\infty$ is K\"{a}hler, $\Gamma\subset U(m)$. Therefore they are all rational singularities, see e.g. Page 50 of \cite{vanishbook92}. Then Theorem \ref{Donldsonsuncoverge} indicate that $M_\infty$ is a normal projective-algebraic variety with at worst rational singularities.

Because $H^2(S^{2m},\mathbb{Z})=0$ for $m>1$, all complex line bundles on $S^{2m}$ are trivial e.g. \cite{holobundle96book}, and $\mathbb{C}^m$ and $\mathbb{C}^m/\Gamma$ with the standard Euclidean metrics are flat, it is easy to see that
\begin{eqnarray*}
\bar{\chi}(\mathbb{C}^m,A_{\omega_{b_l,k}})= \bar{\chi}(X_{\alpha,k,a_l},A_{\omega^{a_l,\alpha,k}})= \bar{\chi}(X_{\gamma,k,a_l},A_{\omega^{a_l,\gamma,k}}) =0,
\end{eqnarray*}
where $X_{\alpha,k,a_l}$ is $\mathbb{C}^m$, and $X_{\gamma,k,a_l}$ is of the form $\mathbb{C}^m/\Gamma$ for some nontrivial finite group $\Gamma \subset U(m)$. Therefore the corresponding quantities $\bar{\chi}(X,A)$ of certain kinds of bubbles vanish in Theorem \ref{highdimthm}.
\end{rem}

Next we state another application of the corresponding compactness theory in higher dimensions which may be of some interest.
For a holomorphic line bundle $L$ on a compact K\"{a}hler $m$-fold $X$, one defines the
volume of $L$ as
\begin{eqnarray*}
v(L)=\limsup_{k\rightarrow \infty}\frac{m!}{k^m}dim H^0(X, L^k),
\end{eqnarray*}
see e.g. \cite{Bouck2,Bouck4} for related results of this aspect.
When $L$ is big or merely nef (numerically
effective), we have (see e.g. \cite{Bouck2})
\begin{eqnarray*}
v(L)=\int_{X} (c_1(L))^m.
\end{eqnarray*}
\begin{thm}\label{volquant}
Let $(M_j, J_j, g_j)$ be a sequence of K\"{a}hler-Einstein $m$ $(>1)$ (complex) dimensional manifolds with uniformly bounded Einstein constants and diameters, uniform $L^m$ Riemannian curvatures bounds, and positive uniform volume lower bounds. Assume that $L_j$ is a holomorphic line bundle over $M_j$, and $A_j$ is a sequence of Hermitian-Yang-Mills connection in $L_j$ with uniform $L^m$ Yang-Mills curvatures bounds. If all $L_j$ are big (or nef), then we have
\begin{eqnarray*}
\lim_{j\rightarrow \infty}v(L_j)&=&\int_{M_\infty} \left(\frac{i}{2\pi}F_{A_\infty}\right)^m +\sum_{l=1}^{u}\sum_{k=1}^{N_{a_l,\beta}}\bar{v}(X_{\beta,k,a_l},A_{\omega^{a_l,\beta,k}}),
\end{eqnarray*}
where $\bar{v}(X,A)$ is the integral of the corresponding $2m$-form $(\frac{i}{2\pi}F_{A})^m$ for the connection $A$ over the underlying space $X$, $(\frac{i}{2\pi}F_{A})^m$ is the wedge product power of $\frac{i}{2\pi}F_A$. Here $M_\infty$ is a K\"{a}hler-Einstein $m$-orbifold, $X_{\beta,k,a_l}$ is a Ricci flat K\"{a}hler ALE space similar to those in Theorem \ref{degeymenergyintr}.
\end{thm}
Theorem \ref{volquant} is a direct consequence of the quantization of $L^m$ Yang-Mills curvature energy in
Theorem \ref{degeymenergyintr} for $m=2$ and in Theorem \ref{degeymenergyhd} for the higher dimensional case $m>2$. Because $H^2(S^{2m},\mathbb{Z})=0$ for $m>1$, all complex line bundles on $S^{2m}$ are trivial (see e.g. \cite{holobundle96book}), it is easy to deduce that
\begin{eqnarray*}
\bar{v}(\mathbb{C}^m,A_{\omega_{b_l,k}})= \bar{v}(X_{\alpha,k,a_l},A_{\omega^{a_l,\alpha,k}})= \bar{v}(X_{\gamma,k,a_l},A_{\omega^{a_l,\gamma,k}}) =0,
\end{eqnarray*}
where $X_{\alpha,k,a_l}$ is $\mathbb{C}^m$, and $X_{\gamma,k,a_l}$ is of the form $\mathbb{C}^m/\Gamma$ for some nontrivial finite group $\Gamma \subset U(m)$,
see the arguments in the above. This is the reason why the corresponding quantities $\bar{v}(X,A)$ of certain kinds of bubbles vanish in Theorem \ref{volquant}.

\subsection{Proof of Theorem \ref{mainaplintr11}}\label{mainappliprf}
It is a direct consequence of the following more general result.
\begin{thm}\label{mainappl444}
Let $(M_j, J_j,g_j)$ be a sequence of K\"{a}hler-Einstein surface as in Theorem \ref{intromainapl33}, and let $A_{j}^{\otimes p}$ be the connection in $K_{M_j}^{-p}$ ($p\geq1$) induced from the metric $g_j$ when the scalar curvature is positive, and let $A_{j}^{\otimes p}$ be the connection in $K_{M_j}^{p}$ induced from the metric $g_j$ when the scalar curvature is negative. If $(M_j, J_j,g_j)$ converges (up to a subsequence) to a orbifold $(M_\infty,J_\infty, g_\infty)$ with singularities $a_1,\cdots,a_u$  (modular bubbles), then
\begin{eqnarray*}
\sum_{l=1}^{u}\mu_{M_{\infty},a_l}(K_{M_{\infty}}^{-p}) &=& \sum_{l=1}^{u}\sum_{k=1}^{N_{a_l,\beta}}\bar{\chi}(X_{\beta,k,a_l},A_{\omega^{a_l,\beta,k}}^{\otimes p})\\
&=&\sum_{l=1}^{u}\sum_{k=1}^{N_{a_l,\beta}}\int_{X_{\beta,k,a_l}}\frac{1}{12}\cdot\frac{1}{8\pi^2}|Rm|^2dvol.\nonumber
\end{eqnarray*}
for $p\geq 1$ when the Ricci curvature is positive,
where $A_{\omega^{a_l,\beta,k}}^{\otimes p}$ are the corresponding bubble connection in $K_{X_{\beta,k,a_l}}^{-p}$ over
the ALE bubble space $X_{\beta,k,a_l}$, $\mu_{M_\infty,a_l}(L_\infty)$ is a rational number related to the local fundamental group.

On the other hand, there holds
\begin{eqnarray*}
\sum_{l=1}^{u}\mu_{M_{\infty},a_l}(K_{M_{\infty}}^{p}) &=& \sum_{l=1}^{u}\sum_{k=1}^{N_{a_l,\beta}}\bar{\chi}(X_{\beta,k,a_l},A_{\omega^{a_l,\beta,k}}^{\otimes p})\\
&=&\sum_{l=1}^{u}\sum_{k=1}^{N_{a_l,\beta}}\int_{X_{\beta,k,a_l}}\frac{1}{12}\cdot\frac{1}{8\pi^2}|Rm|^2dvol.\nonumber
\end{eqnarray*}
for $p\geq 2$ when the Ricci curvature is negative, where $A_{\omega^{a_l,\beta,k}}^{\otimes p}$ are the corresponding bubble connection in $K_{X_{\beta,k,a_l}}^{p}$ over
the ALE bubble space $X_{\beta,k,a_l}$.
\end{thm}

\begin{proof}
Note that $M_j$ in the theorem is a projective surface \cite[Chapter 5]{Huybrechts05book}, and $K_{M_\infty}^{-p}$ is ample for every $p\geq 1$ when the scalar curvature is positive, $K_{M_\infty}^{p}$ is ample for every $p\geq 1$ when the scalar curvature is negative. And moreover $A_{j}^{\otimes p}$ is a Hermitian-Yang-Mills connection.

Firstly, we prove the theorem in the positive Ricci curvature case. By Theorem \ref{tianconvegthm} and Theorem \ref{Donldsonsuncoverge}, we know that $K_{M_\infty}^{-p}$ is ample for every $p\geq 1$.
So we have
\begin{eqnarray*}
\lim_{j\rightarrow \infty}dim H^{0}(M_j;K_{M_j}^{-p})= dim H^{0}(M_\infty;K_{M_\infty}^{-p}) &-& \sum_{l=1}^{u}\mu_{M_{\infty},a_l}(K_{M_\infty}^{-p}) \\
&+& \sum_{l=1}^{u}\sum_{k=1}^{N_{a_l,\beta}}\bar{\chi}(X_{\beta,k,a_l},A_{\omega^{a_l,\beta,k}}^{\otimes p}),
\end{eqnarray*}
from (\ref{h0111}) in  Theorem \ref{intromainapl33} by noticing that $K_{M_\infty}\otimes K_{M_\infty}^{-(1+p)}=K_{M_\infty}^{-p}.$
By Lemma 5.2 of \cite{Tian1990} (see also Proposition 2.1 in \cite{tian13c0} and pages 85-86 of \cite{DonaldsonSunacta}), we have that
\begin{eqnarray*}
lim_{j\rightarrow \infty} \ dimH^0(M_j; K_{M_j}^{-p})=dimH^0(M_\infty; K_{M_\infty}^{-p})
\end{eqnarray*}
for any integer $ p > 0$. Therefore, the first identity in this case follows.

By the fact that the Ricci curvature of the ALE space $(X_{\beta,k,a_l},h_{\beta,k,a_l})$ in the above is zero, so the corresponding curvature $F_{\omega^{a_l,\beta,k}}^{\otimes p}$ of $A_{\omega^{a_l,\beta,k}}^{\otimes p}$ is zero, and hence
we have by (\ref{importformula111}) that
\begin{eqnarray*}\label{importformula11}
&&\bar{\chi}(X_{\beta,k,a_l},A_{\omega^{a_l,\beta,k}}^{\otimes p})\\
&=&\int_{X_{\beta,k,a_l}}\frac{1}{12}c_2(X_{\beta,k,a_l},h_{\beta,k,a_l})
+\frac{1}{2}\left(c_1(K_{X_{\beta,k,a_l}}^{-p},A_{\omega^{a_l,\beta,k}}^{\otimes p})\right)^2\\
&=&\int_{X_{\beta,k,a_l}}\frac{1}{12}\cdot \frac{1}{8\pi^2}|Rm|^2dvol
-\frac{1}{8\pi^2}\int_{X_{\beta,k,a_l}} \left(F_{\omega^{a_l,\beta,k}}^{\otimes p}\wedge F_{\omega^{a_l,\beta,k}}^{\otimes p}\right)\\
&=&\int_{X_{\beta,k,a_l}}\frac{1}{12}\cdot\frac{1}{8\pi^2}|Rm|^2dvol.
\end{eqnarray*}
So the second identity in this case holds. In the above, we have used the fact that for any connection $A$ in a holomorphic line bundle $L$ over $X$, it holds
\begin{eqnarray*}
\frac{1}{2}(c_1(L,A))^2=-\frac{1}{8\pi^2}F_A\wedge F_A.
\end{eqnarray*}

Secondly, in the negative Ricci curvature case, we can use results in \cite[pages 85-86]{DonaldsonSunacta} to show that
\begin{eqnarray*}
lim_{j\rightarrow \infty} \ dimH^0(M_j; K_{M_j}^{p})=dimH^0(M_\infty; K_{M_\infty}^{p})
\end{eqnarray*}
for $p\geq 1$. Since $K_{M_\infty}^{p}$ with $p\geq 1$ is ample, we can get
the first identity in this case from (\ref{h0111}) in Theorem \ref{intromainapl33} for $K_{M_\infty}\otimes K_{M_\infty}^{p} =K_{M_\infty}^{(1+p)}.$ Note that $p+1\geq 2$ here. The second identity in this case follows from applying similar argument as in the positive scalar curvature case.
\end{proof}

Finally, Theorem \ref{mainaplintr11} follows as a corollary.

\begin{proof}[\textbf{Proof of Theorem \ref{mainaplintr11}}]
For the positive scalar curvature case, the first identity has been proved in Theorem \ref{mainappl444}. The second identity follows from the classical observation (\ref{euler22}) and the Riemannian curvature energy identity in Theorem \ref{mainconvgethm}. The case of negative scalar curvature can be proved in a similar way.
\end{proof}

\vskip5pt

\subsection{Application to some aspects of the moduli spaces of Del Pezzo surfaces}\label{moduliappli}
In \cite{Tian1990}, Tian proved the existence of K\"{a}hler-Einstein on smooth Del Pezzo surfaces, and initiated the studies of the corresponding moduli space on the basis of some important classical works e.g. Tian-Yau \cite{tianyau87}. Two decades later, in \cite{sunjdg16} the authors proved that the Gromov–Hausdorff compactification of the
moduli space of K\"{a}hler-Einstein Del Pezzo surfaces (which have positive scalar curvature) in each degree
agrees with certain algebro-geometric compactification. In particular, by \cite{Tian1990,sunjdg16} (and some other related works, see references in \cite{sunjdg16}) it was shown that the Gromov–Hausdorff limit of a
sequence of K\"{a}hler-Einstein Del Pezzo surfaces is a Kähler–Einstein log
Del Pezzo with either canonical (ADE) singularities or cyclic quotient
singularities of certain types. For more related results about moduli spaces of $\mathbb{Q}$-Gorenstein smoothable Fano varieties, we refer to e.g. \cite{SSY-exis,SS-moduli,Od-moduli,LWC-moduli}.

In the following, we will recall some related results (see e.g. the introduction and section 2 of \cite{sunjdg16}). Since it is well know that for $d \geq 5 $ the moduli space is
just a single point, we can assume the degree $d \in \{1,2,3,4\}$. When $d=3,4$, the Gromov–Hausdorff limit $M_{\infty} $ have only canonical singularities of type $A_1$ or $A_2$. When $d = 2$, $M_{\infty} $ can have only $A _1 , A _2 , A_ 3 , A _4 $, and $\frac{1}{4} (1,1) $ type cyclic quotient singularities.
 When $d = 1$, $M_{\infty} $ can have only $\frac{1}{4} (1,1) $, $\frac{1}{8} (1,3) $, $\frac{1}{9} (1,2) $
 singularities
besides $A_i $ $(i \leq 8)$ and $D_ 4$ singularities (see Lemma 5.16 of \cite{sunjdg16}). An important new result in \cite{sunjdg16} is that the class of log Del Pezzo surfaces with canonical
singularities is not sufficient to construct a K\"{a}hler-Einstein moduli variety. In particular, they found some $\mathbb{Q}$-Gorenstein smoothable Kähler–Einstein log
Del Pezzo surfaces with non-canonical singularities. In other words, certain types of cyclic quotient singularities will appear definitely in some Gromov–Hausdorff limit $M_{\infty} $ when $d=1,2$.

To get further applications on the basis of Theorem \ref{mainaplintr11}, we need to compute the quantities $\mu$ which are rational numbers related to the structure of the line bundles at the singularities. It is a difficult task. But fortunately they was calculated in \cite{ChTs-BMY,LimRota22,LimRota23} for canonical singularities. In fact, for a complex 2-orbifold $X$ the line bundle $K_X^{-1}$ is locally equivalent to $\mathcal{O}_X$ at the canonical singularities (the local fundamental groups $\Gamma$ are in $SU(2)$), we can apply the results in \cite{ChTs-BMY} directly. We transcribe some conclusions that will be used here. For singularities of type $A_k$, we have $$12\mu_{X,p}(K_X^{-1})=(k+1)-\frac{1}{k+1}.$$
For singularities of type $D_4,$ we have 
$$12\mu_{X,p}(K_X^{-1})= \frac{39}{8}.$$

 And we are able to compute the quantities $\mu$ by {\bf Dedekind sums} (see e.g. \cite{ReidYPS, Icecream}) if the singularities are of type $\frac{1}{r} (b_1,b_2) $, where $r,b_1,b_2$ are integers. The $i$-th Dedekind sum $\sigma_i$ defined by

\begin{eqnarray*}
\sigma_{i}\left(\frac{1}{r}(b_{1}, \ldots, b_{m})\right) = \frac{1}{r} \sum_{\substack{\varepsilon \in \mu_{r} \\ \varepsilon^{b_{j}} \neq 1, \ \forall j=1,\ldots,m}} \frac{\varepsilon^{i}}{(1 - \varepsilon^{b_{1}}) \cdots (1 - \varepsilon^{b_{m}})},
\end{eqnarray*}
where $\varepsilon$ runs over the $r$-th roots of unity for which the denominator is
nonzero. By the arguments in \cite{ReidYPS} (page 407) and \cite{Icecream} (Theorem 3.3)
\begin{eqnarray*}
\mu_{X,p}(L)=\sigma_{i}\left(\frac{1}{r}(b_{1}, b_{2})\right)
\end{eqnarray*}
 if the line bundle $L$ is locally of type $ _i\left(\frac{1}{r}(b_{1}, \ldots, b_{m})\right)$ at the singularity $p$ (i.e. it is locally equivalent to the eigensheaf $\mathcal{L}_i$ ). We remark that it looks a little different from the formula in \cite{ReidYPS,Icecream} where the correction term is
\begin{eqnarray*}
c_p(D)=\sigma_{i}\left(\frac{1}{r}(b_{1}, b_{2})\right)-\sigma_{0}\left(\frac{1}{r}(b_{1}, b_{2})\right)
\end{eqnarray*} for the reason that the authors had already put the correction term $\sigma_{0}\left(\frac{1}{r}(b_{1},  b_{2})\right)$ corresponding to the eigensheaf $\mathcal{L}_0$ into $\chi(\mathcal{O}_X)$.
More precisely, the difference in the above comes from the different formulas
 \begin{eqnarray*}
\chi(X;L)=RR(D)+\sum_{p} c_p(D)
\end{eqnarray*}
and
 \begin{eqnarray*}
\chi(X;L)=\chi_{orb}(X;L)+\sum_{p} \mu_{X,p}(L)
\end{eqnarray*}
 used by different researchers, where $L=\mathcal{O}_X(D)$. A crucial term of $RR(D)$ is $\chi(\mathcal{O}_X)$, see e.g. Theorem 3.2 of \cite{Icecream}.
For more details, we refer to Chapter 3 of \cite{ReidYPS} and \cite{Icecream}.

 And again arguments in \cite{ReidYPS}(page 409) show that $\mu_{M_{0},p}(K_{M_{0}}^{-1})$ is of type $ _k\left(\frac{1}{r}(b_{1}, \ldots, b_{m})\right)$ if $p$ is a type $\frac{1}{r}(b_{1}, \ldots, b_{m})$ singularity, where $M_0$ is a $m$-dimensional complex orbifold, $-k=b_1+\cdots,b_m$. Then by the obvious properties of Dedekind sums $\sigma_i=\sigma_{i+r}$, we have that
 \begin{eqnarray*}
\mu_{M_{0},p}(K_{M_{0}}^{-1})=\sigma_{r+k}\left(\frac{1}{r}(b_{1}, \ldots, b_{m})\right),
\end{eqnarray*}
where $k=-(b_1+\cdots,b_m)$ is given as above.

By direct computations (with the use of the software Wolfram Mathematica), we get that
 \begin{eqnarray*}
 \sigma_{2}\left(\frac{1}{4}(1,  1)\right)&=&\sigma_{0}\left(\frac{1}{4}(1,  1)\right)=\frac{1}{16},\\
 \sigma_{4}\left(\frac{1}{8}(1,  3)\right)&=&\sigma_{0}\left(\frac{1}{8}(1,  3)\right)=\frac{5}{32},\\
 \sigma_{6}\left(\frac{1}{9}(1,  2)\right)&=&\sigma_{0}\left(\frac{1}{9}(1,  2)\right)=\frac{2}{27}.
 \end{eqnarray*}
Therefore
\begin{eqnarray*}
12\mu_{X,p}(K_{X}^{-1})=\frac{3}{4},\,\frac{15}{8}\, \text{or}\,\frac{8}{9}
\end{eqnarray*}
respectively,
if $p$ is a singularity of type $\frac{1}{4} (1,1) $, $\frac{1}{8} (1,3) $, or $\frac{1}{9} (1,2) $. Before the proof of Theorem \ref{mainaplintr22}, we recall the quantities $\mu$ for some canonical singularities.
\begin{eqnarray*}
12\mu_{X,p}(K_{X}^{-1})=\frac{3}{2},\,\frac{8}{3},\,\frac{15}{4},\,\frac{24}{5},\,
\cdots,\,9-\frac{1}{9},\, \text{or}\,\frac{39}{8}
\end{eqnarray*}
respectively,
if $p$ is a singularity of type $A_1,\cdots,A_8$ or $D_4$.

\begin{proof}[\textbf{Proof of Theorem \ref{mainaplintr22}}]
Notice that for the degree $d$ case, $\chi(M_j)=12-d.$
So from Theorem \ref{mainaplintr11} we have
\begin{eqnarray}\label{lericompa}
12\sum_{l=1}^{u}\mu_{M_{\infty},a_l}(K_{M_{\infty}}^{-1})<12-d,
\end{eqnarray}
since the sequence $M_j$ is non-collapsing and $M_\infty$ is non-flat.

First we prove (a) of Theorem \ref{mainaplintr22}. In this case, $d=3$, so
\begin{eqnarray*}\label{lericompa}
12\sum_{l=1}^{u}\mu_{M_{\infty},a_l}(K_{M_{\infty}}^{-1})<9,
\end{eqnarray*}
and then by using the equality $12\mu_{X,p}(K_{X}^{-1})=\frac{3}{2}$ for $A_1$ singularities, it is easy to see that the number of $A_1$ singularities is $\leq 5$.

For (b) and (c), we only prove that the number of $A_4$ singularities is at most 1. In this case $\chi(M_j)=10.$ If we assume that $M_\infty$ has two $A_4$ singularities, then we have that
\begin{eqnarray*}
10>12\sum_{l=1}^{u}\mu_{M_{\infty},a_l}(K_{M_{\infty}}^{-1})\geq 10-\frac{2}{5}=2\times \frac{24}{5}.
\end{eqnarray*}
It follows that $M_\infty$ has no other singularities, since $\frac{2}{5}$ is smaller than any $12\mu$ listed in the above. That is to say $M_\infty$ has only canonical singularities. But by the results in \cite{sunjdg16} (see Page 165), for $d=2$, $M_\infty$ can not have $A_4$ singularities if it has only canonical singularities. Thus we get a contradiction. Therefore $M_\infty$ has at most one $A_4$ singularity.

For (d), firstly, we prove that the number of singularities of type $D_4$ is not larger than 2. In this case $\chi(M_j)=11.$
Hence \begin{eqnarray*}
12\sum_{l=1}^{u}\mu_{M_{\infty},a_l}(K_{M_{\infty}}^{-1})<11,
\end{eqnarray*}
and then by noticing $12\mu_{X,p}(K_{X}^{-1})=\frac{39}{8}$ for $D_4$ singularities, it is easy to see that the number of $D_4$ singularities is $\leq 2$.

Then we prove that if $M_\infty$ has two $D_4$ singularities, then it has at most a $\frac{1}{4}(1,1)$ singularity or a $\frac{1}{9}(1,2)$ singularity besides these two singularities.
Let $a_1$ and $a_2$ be two $D_4$ singularities, then
\begin{eqnarray*}
12\sum_{l=3}^{u}\mu_{M_{\infty},a_l}(K_{M_{\infty}}^{-1})<11-2\times(5-\frac{1}{8})=\frac{5}{4}.
\end{eqnarray*}
Noticing that all the $12\mu$ listed in the above are bigger than $\frac{5}{4}$ besides $12\mu=\frac{3}{4}$ or $12\mu=\frac{8}{9}$ for $\frac{1}{4}(1,1)$ or $\frac{1}{9}(1,2)$ singularities, we can obtain the conclusion.

The others can be deduced by similar arguments in the above. The proof is finished.
\end{proof}

\begin{exam}\label{exam11}
Let the surfaces $X_s$ be the double cover of $\mathbb{P}(1,1,4)$ branched along the hyper-elliptic curve $z_3^2=f_8(z_1,z_2) $ which is denoted by $C$ (see Theorem 5.10 and Remark 5.14 of \cite{sunjdg16}), it has exactly two $\frac{1}{4}(1,1)$ singularities, and it is of degree $2$. Because $X_s$ is in the moduli space constructed in \cite{sunjdg16}, so it is the limit $M_\infty$ of a sequence of smooth K\"{a}hler-Einstein Del Pezzo surface $M_j$ of degree $2$.

Firstly, we compute the orbifold Euler number by methods of algebraic topology. It is easy to see that $\mathbb{P}(1,1,4)$ has same Euler number as $\mathbb{P}^2$, which is $3$. The curve $C$ has genus $3$ by Theorem 8.3 of \cite{AGGfib}(see also Corollary 3.5 of \cite{OW-top1} and Theorem 5.3.7 of \cite{weigpro}),
which states that the genus of a non-singular algebraic curve of degree $\bar{d}$ in the weighted
projective plane $\mathbb{P}(a_0, a_1, a_2)$ is given by
\begin{equation*}
g=\frac{1}{2}\left(\frac{\bar{d}^2}{a_0a_1a_2} -\bar{d}\sum_{i<j}\frac{gcd(a_i,a_j)}{a_ia_j} +\sum_{i}\frac{gcd(a_i,\bar{d})}{a_i}-1\right).
\end{equation*}
Note that in this case $\bar{d}=8$.
So the Euler number of
$C$ is $2-2g=2-6=-4$. Hence we have that $$\chi(X_s) =2\chi(P(1,1,4))-\chi(C)=2\times 3-(-4)=10. $$
Then the orbifold Euler number is
\begin{equation}\label{orbeuler1}
\chi_{orb}(X_s) =\chi(X_s)-\sum_{p }(1-\frac{1}{n_p})=10-2\times \frac{3}{4}=8+\frac{1}{2},
\end{equation}
where $n_p$ is the order of the local fundamental group at the singularity $p$. We remark that the order of the local fundamental group at a singularity of type $\frac{1}{r}(b_1,b_2)$ is $r$.

Secondly, we try to get the orbifold Euler number by applying Theorem \ref{mainaplintr11}. Notice that in this case the degree $d$ is $2$, so $\chi(M_j)=10.$
So from Theorem \ref{mainaplintr11} we have
\begin{eqnarray*}
\chi_{orb}(X_s)
=10-12\sum_{l=1}^{2}\mu_{M_{\infty},a_l}(K_{M_{\infty}}^{-1})=10-2\times \frac{3}{4}=8+\frac{1}{2},
\end{eqnarray*}
by noticing that $12\mu_{X,p}(K_{X}^{-1})=\frac{3}{4}$ for $\frac{1}{4}(1,1)$ singularities.
Therefore we conclude that we can obtain the same result by using the two methods for this example.
\end{exam}

\begin{exam}\label{exam22}
Let $X_1^{T}$
 be the quotient of $\mathbb{P}^2$ by $\mathbb{Z}/9\mathbb{Z}$ (see Example 3.10 and Example 5.8 of \cite{sunjdg16}), where the
generator $\xi$ of $\mathbb{Z}/9\mathbb{Z}$ acts by $$\zeta_9\cdot[z_1 : z_2 : z_3]=[z_1 : \zeta_9z_2 : \zeta_9^{-1}z_3]$$
and $\zeta_9$ is the primitive ninth root of unity. Then $X_1^{T}$ is a degree one
log Del Pezzo surface, with one $A_8$ singularity at $[1 : 0 : 0]$ and two
$\frac{1}{9}(1, 2)$ singularities at $[0 : 1 : 0] $ and $[0 : 0 : 1]$. And $X_1^{T}$ is the limit $M_\infty$ of a sequence of smooth K\"{a}hler-Einstein Del Pezzo surface $M_j$ of degree $1$.

Let $g_0$ be the standard Fubini-Study metric on $\mathbb{P}^2$, and $\Pi$ be the natural map from
$\mathbb{P}^2$ to $X_1^{T}$,
then we have
\begin{eqnarray*}
\chi(\mathbb{P}^2)&=&\frac{1}{8\pi^2}\int_{\mathbb{P}^2}|Rm(g_0)|^2dV_{g_0}\\
&=&9\times\frac{1}{8\pi^2}\int_{X_1^{T}}|Rm(h_0)|^2dV_{h_0}\\
&=&9\chi_{orb}(X_1^{T})=3
\end{eqnarray*}
where $h_0$ is smooth outside the singularities and $g_0=\Pi^*(h_0)$.
That is $\chi_{orb}(X_1^{T})=\frac{1}{3}.$
 Notice that the order of the local fundamental group at a singularity of type $\frac{1}{9}(1,2)$ and $A_8$ is $9$. It is easy to know that $\chi(X_1^{T})=3 $ by the same computations as in \eqref{orbeuler1}.

Next, notice that in this case the degree $d$ is 1, so $\chi(M_j)=11.$
So from Theorem \ref{mainaplintr11} we have
\begin{eqnarray*}
\chi_{orb}(X_1^{T})
=11-12\sum_{l=1}^{3}\mu_{M_{\infty},a_l}(K_{M_{\infty}}^{-1})=11-2\times \frac{8}{9}-(9-\frac{1}{9})=\frac{1}{3},
\end{eqnarray*}
by noticing that $12\mu_{X,p}(K_{X}^{-1})=\frac{8}{9},\,9-\frac{1}{9} $ for $\frac{1}{9}(1,2)$ or $A_8$ singularities.
So again we can obtain the same result by using two different methods for this example.

And one can apply the same arguments in the above to $X_2^{T}$ in Example 5.7. of \cite{sunjdg16} which is the quotient of $\mathbb{P}^1\times \mathbb{P}^1$ by $\mathbb{Z}/4\mathbb{Z}$, and is a degree
$2$ log Del Pezzo surface with two $A_3$ singularities and two $\frac{1}{4}$ singularities.
\end{exam}

Finally we will prove Proposition \ref{HRR-Milnor}.

\begin{proof}[\textbf{Proof of Proposition \ref{HRR-Milnor}}]
Let $M_\infty$ be the limit of a sequence of Del Pezzo surfaces $M_j$ of degree $d$ as in Theorem \ref{mainaplintr22}. By Theorem \ref{mainaplintr11}, we have
\begin{eqnarray*}
\chi_{orb}(M_{\infty})
+12\sum_{p\in Sing(M_{\infty})}\mu_{M_{\infty},p}(K_{M_{\infty}}^{-1})=12-d.
\end{eqnarray*}
And by Proposition 2.6 of \cite{HPdel}, we have
\begin{eqnarray*}
K_{M_{\infty}}^{\wedge 2}+\chi(M_{\infty})+\sum_{p\in Sing(M_{\infty})}\nu_p=12\chi(\mathcal{O}_{M_{\infty}}).
\end{eqnarray*}
Noticing that $K_{M_{\infty}}^{\wedge 2}=d$ (here $K_{M_{\infty}}^{\wedge 2}$ is the intersection number of $K_{M_{\infty}}$ with itself, usually it is just written as $K_{M_{\infty}}^{2}$ in other literatures), and $\chi(\mathcal{O}_{M_{\infty}})=1$ (see e.g. \cite{HPdel}),
we have
\begin{eqnarray*}
\chi(M_{\infty})+\sum_{p\in Sing(M_{\infty})}\nu_p=12-d.
\end{eqnarray*}
And then by the relation between $\chi(M_{\infty})$ and $\chi_{orb}(M_{\infty})$(see \eqref{euler11}), we get
\begin{eqnarray*}
\sum_{p\in Sing(M_{\infty})}(1-\frac{1}{n_p})+\sum_{p\in Sing(M_{\infty})}\nu_p=12\sum_{p\in Sing(M_{\infty})}\mu_{M_{\infty},p}(K_{M_{\infty}}^{-1}).
\end{eqnarray*}

And the second equality follows from the above equation and an identity in the proof of Lemma 5.16 of  \cite{sunjdg16} which states that
\begin{eqnarray*}
K_{M_{\infty}}^{\wedge 2}+\rho(M_{\infty})+\sum_{p\in Sing(M_{\infty})}\nu_p=12\chi(\mathcal{O}_{M_{\infty}})-2.
\end{eqnarray*}
\end{proof}

\

\section{Some preliminaries on non-collapsed Einstein 4-manifolds and Yang-Mills  fields}\label{Preliminaries111}

\vskip5pt

\subsection{Bubble tree convergence of non-collapsed Einstein 4-manifolds}
We shall firstly recall the classical compactness theory of non-collapsed Einstein 4-manifolds developed in
\cite{nakajima1988hausdorff,anderson1989ricci, BKN, Tian1990,  bando1990bubbling, Anderson1992,Anderson94icm, nakajima1994convergence,Jeff2006Curvature} etc.

Let $(M_j ,g_j)$  be a sequence of closed Einstein $4$-manifolds with uniformly bounded Einstein constants $|\mu_j|\leq \mu$ and satisfying (\ref{einstincondintr}). Then $(M_j ,g_j)$ have a priori $L^2$ Riemannian curvature estimates \eqref{cheegernaberjiang}.
By the study in the above mentioned works, we can suppose that $(M_j ,g_j)$ converges smoothly except at the set $S$ of singular points up to a subsequence. Then for every $x_a \in S$, there is a point $x_{a,j}\in M_j$ such that for a positive constant $r_\infty$ sufficiently small,
\begin{equation*}
\sup_{B(x_{a,j},r_\infty)} |Rm({g_j})|^2=|Rm({g_j})|^2(x_{a,j})\rightarrow \infty \quad \text{as} \quad j\rightarrow \infty
\end{equation*}
and
\begin{equation*}
\int_{B(x_a,r_\infty)} |Rm(g_\infty)|^2dV_\infty\leq \frac{\varepsilon}{2}
\end{equation*}
with a small positive number $\varepsilon<\bar{\varepsilon}$, where $\bar{\varepsilon}>0$ is determined by the small energy regularity theorem for Einstein metric (see e.g. \cite{anderson1989ricci,Tian1990,Jeff2006Curvature}). For sufficiently large $j$, we can find a positive number $r_j>0$ so that
\begin{equation*}
\int_{B(x_{a,j},r_\infty)\setminus B(x_{a,j},r_j)} |Rm(g_j)|^2dV_j =\varepsilon.
\end{equation*}
It is easy to see that
$r_j\rightarrow 0 $ as $j\rightarrow \infty.$

We remark that if moreover we assume that
\begin{equation*}
\int_{M_j} |Rm(g_j)|^{n/2}dV_j \leq R,
\end{equation*}
for some finite $R>0$, where $n$ is the real dimension of $M_j$, then we can extend the corresponding results in the following theorem to the higher dimensional situations. In particular, a similar Riemannian Curvature estimate for the neck regions and a quantization result for the $L^{n/2}$ Riemannian curvature energy hold.

\begin{thm}\label{mainconvgethm}(\cite{nakajima1988hausdorff,anderson1989ricci, BKN, Tian1990,  bando1990bubbling, Anderson1992,Anderson94icm, nakajima1994convergence,Jeff2006Curvature,Anderson94icm,cheeger2015regularity} etc.) Let $(M_j ,g_j)$ be a sequence of closed Einstein 4-manifolds as above.
 Then there exist a subsequence of $\{j\}$ (still denoted by $j$) and a compact Einstein orbifold $(M_\infty, g_\infty )$ with a finite set (possibly empty) of orbifold singular points $S=\{x_1,x_2,\cdots,x_s\}\subset M_\infty$  for which the following statements hold:

\begin{itemize}
\item[(1)] $(M_j ,g_j)$ converges to $(M_\infty, g_\infty )$ in the Gromov-Hausdorff distance. There exists an into diffeomorphism $F_j:M_\infty\rightarrow M_j$ for each $j$ such that $F^{*}_{j}g_j$ converges to
$g_\infty$ in the $C^\infty$-topology on $M_\infty \setminus S$.

\item[(2)] For every $x_a\in S$ ($a=1,2,\cdots,s$) and $j$, let $x_{a,j}\in M_j$ and $r_j$ be chosen in the above way, then
\begin{itemize}
\item[(2.a)] $B(x_{a,j} ,\delta)$ converges to $B(x_a, \delta) $ in the Gromov-Hausdorff distance for all $\delta>0$.

\item[(2.b)] Up to a subsequence, $((M_j,r_j^{-2}g_j),x_{a,j})$ converges to an \textbf{ALE bubble space} $((Y,h),y_\infty)$ in the pointed Gromov-Hausdorff distance, where $(Y,h)$ is a complete, Ricci flat, non-flat ALE 4-orbifold with a finite set (possibly empty) of orbifold singularities. The convergence is smooth outside the singular points.
\item[(2.c)] (\textbf{Riemannian Curvature estimate}) There exist positive constants $C_7>0$ and $\varepsilon_5>0$ such that for $ 4r_j\leq r< 4r \leq r_\infty$ it holds that
\begin{equation*}\label{neckcurvturestmprop}
r^2|Rm({g_j})|\leq C_7 \max \left\{(\frac{r_j}{r})^{\varepsilon_5},(\frac{r}{r_\infty})^{\varepsilon_5}  \right\}.
\end{equation*}

\end{itemize}

\item[(3)]In the case that there is only one ALE bubble manifold at each singular point, we have that
\begin{itemize}
\item[(3.a)] $B(x_{a,j} ,\delta)$ converges to $B(x_a, \delta) $ in the Gromov-Hausdorff distance for all $\delta>0$.

\item[(3.b)] $(( M_j , r_{j}^{-2}g_j),x_{a,j})$ converge to $(( M_a , h_a),x_{a,\infty})$ in the pointed Gromov-Hausdorff distance, where $( M_a , h_a)$ is a Ricci flat ALE $4$-manifold.

\item[(3.c)] There exists an into diffeomorphism $G_j : M_a \rightarrow M_j$ such that $G^{*}_{j}(r_{j}^{-2} g_j)$ converges to $h_a$ in the $C^\infty$-topology on $M_a$.

\end{itemize}
\item[(4)] If there are several bubble manifolds (or orbifolds) at some singular points, by repeating the same process as above, we can obtain a bubble tree $\mathbf{Tr_{ALE}}$ consisting of a finite number of Ricci flat ALE bubble spaces.

\item[(5)] (\textbf{Curvature energy identity}) \label{curenergyidentity}
Let $r_j,r_\infty,x_{a,j}$ be as before, then it holds
\begin{equation*}
\lim_{K_2\rightarrow \infty}\lim_{K_1\rightarrow \infty }\lim_{r_j\rightarrow 0}\int_{B(x_{a,j},K_2^{-1}r_\infty)\setminus B(x_{a,j},K_1r_j)}|Rm(g_j)|^2 dV_j=0.
\end{equation*}
Or equivalently, it holds that
\begin{equation*}
\lim_{j\rightarrow \infty}\int_{M_j}|Rm(g_j)|^2 dV_j= \int_{M_\infty}|Rm(g_\infty)|^2 dV_\infty +\sum_k \int_{X_k}|Rm(h_k)|^2 dV_{h_k},
\end{equation*}
where $\{(X_k,h_k)_{k=1}^N\}$ is the set of all ALE bubble spaces in the tree $\mathbf{Tr_{ALE}}$ above.

\end{itemize}

\end{thm}

\begin{rem}\label{remorbfsmoth}
By arguments in \cite{anderson1989ricci,BKN,Tian1990} (see e.g. Section 4 of \cite{Tian1990} or Theorem (5.1) of \cite{BKN}), the metric $g_\infty$ on $M_\infty$ extends smoothly across $x_a$ as an orbifold Einstein metric. That is to say, there is a (covering) map
$\Pi_a:B_1\setminus\{0\}\rightarrow B(x_a, \delta)\setminus\{x_a\} $ such that
$\Pi^{*}_{a}g_\infty $ extends to a smooth Einstein metric on $B_1\subset \mathbb{R}^4$.
\end{rem}

\begin{rem} \label{neckprop}(see e.g. \cite{bando1990bubbling})
If one takes $1<K_1<K_2$ sufficiently large, then the degenerate neck region $B(x_{a,j},K_2^{-1}r_\infty)\setminus B(x_{a,j},K_1r_j)$ is close to a potion of some flat cone $\mathbb{R}^4/\Gamma$ for large $j$.
\end{rem}

\subsection{Blow-up analysis of Yang-Mills fields}
We recall some basic knowledge and the compactness theory of Yang-Mills fields, see e.g. \cite{UhlenbeckRemovablesingu, UhlenbeckLp,Instanton4manifolds,geofourfolds,1982Gauge,YMcompbook}.

Let $(M,g)$ be a Riemannian 4-manifold and $E$ be some vector
bundle over $M$ with structure group $G\subset SO(r)$ whose Lie algebra is denoted by $\mathfrak{g}$.
A connection $A$ in $E$ induces covariant derivatives $\nabla_A$ on various vector
bundles related to $E$. We denote them by $\nabla_{g,A} $ when we want to emphasize the role of the metric $g$. $D_A=\wedge\circ \nabla_A$ gives the exterior differential on vector-valued forms.
The curvature $F_A$ of a connection $D_A$ is define as
$$D_A\circ D_A:\Omega^{*}(E)\rightarrow \Omega^{*+2}(E).$$
The Yang-Mills equation (\ref{ymequintro}) is Euler-Lagrange equation of the functional defined in
(\ref{ymfunct}).
By the Bianchi identities $D_AF_A=0$, we have that from (\ref{ymequintro}) that
\begin{equation*}
\triangle_{A}F_A=D_AD^{*}_AF_A+D^{*}_AD_AF_A=0
\end{equation*}
for any Yang-Mills connection $A$.

In the following, we collect some basic geometric analysis results developed for blow-up theory of Yang-Mills fields in dimension 4.

\begin{thm}\label{smallenergythm}($\epsilon$-regularity, see e.g. \cite{UhlenbeckRemovablesingu,UhlenbeckLp,geofourfolds,YMcompbook})
 Let $(B_1,g) $ be a ball in $\mathbb{R}^4$ equipped with a smooth metric $g$ which is close to the Euclidean metric in the sense that there is some sufficiently small positive number $\eta<1$ such that in local coordinates
  \begin{equation*}
|| g_{kl}-\delta_{kl}||_{C^4(B_1)}<\eta.
\end{equation*}
Let $G$ be a compact Lie group with Lie algebra $\mathfrak{g}$. There exists
$\epsilon_G>0$ and $C_G>0$ such that for any connection which is $d+A$ in local trivialization, where $A \in  W^{1,2}(T^1B_1 \otimes\mathfrak{g})$ satisfying
\begin{equation*}\label{smallcondit}
\int_{B_1}|F_A|^2dV_g\leq \epsilon_G,
\end{equation*}
there exists $h \in W^{2,2} (B_1 ,G)$ such that
\begin{eqnarray*}\label{smallestym}
& &\int_{B_1}|A^h|^2+|\nabla_g A^h|^2dV_g\leq C_G\int_{B_1}|F_A|^2dV_g,\nonumber\\
& &d^{*}A^h=0 \quad \text{in  } B_1,\\
& &\imath^{*}_{\partial B_1}(*A^h)=0\nonumber,
\end{eqnarray*}
where $A^h=h^{-1}dh+h^{-1}Ah$, and $\imath_{\partial B_1} $ is the canonical inclusion map of the boundary of $B_1$ into $\mathbb{R}^4$.
\end{thm}

\begin{thm}(Removable singularity, see e.g. \cite{UhlenbeckRemovablesingu,UhlenbeckLp})\label{remsingu}
Let $(B_1,g)$ be as in Theorem \ref{smallenergythm}, and let $ d+A$ be a Yang-Mills connection in a bundle $E$ with a compact structure group $G$ over $B_1 \setminus  \{0\} $,
if
\begin{equation*}
\int_{B_1}|F_A|^2dV_g\leq \infty,
\end{equation*}
then there is a gauge $h\in W^{2,2}_{loc} (B_{1} ,G) $ in which the bundle $E $ can be extended to a smooth  bundle $ \bar{E}$ over $B_{1}$ and the connection $A^h$ extends
to a smooth Yang-Mills connection $\bar{A}$ in $ \bar{E}$.
\end{thm}

We remark that all Riemannian metrics are close to the Euclidean metrics at small scales, so it is easy to see that we can apply above theorems to the study of blow-up analysis of Yang-Mills fields on Riemannian 4-manifolds by noticing that the Yang-Mills functional is scaling invariant in this case.

\begin{thm}\label{energygap}
(Energy gap, see e.g. \cite[Chapter 2]{geofourfolds}, \cite{1982Gauge} or \cite{2010An}) Let $ A$ be a Yang-Mills connection in a bundle $E$ with structure group $G$ over $\mathbb{R}^4 / \Gamma$, here $\Gamma $ is some finite group in $SO(4)$. Then there exists $\epsilon_G > 0$ depending only on $G$ such that if
\begin{equation*}\label{smallcondit2}
\int_{\mathbb{R}^4/ \Gamma}|F_A|^2dx\leq \epsilon_G,
\end{equation*}
then $A$ is a flat connection.
\end{thm}
\begin{proof}
In the case of $\Gamma=\{e\}$, this is well known fact. In the case of a nontrivial group $\Gamma$, we can lift the connection $A$ to be a connection $\tilde{A}$ in some bundle $\tilde{E}$ over $\mathbb{R}^4 $. Then classical arguments will work.
\end{proof}

Finally, we recall the classical energy quantization theorem for Yang-Mills fields over a fixed Riemannian 4-manifold. For simplicity of notations, we state the result in the case that there is only one bubble.

\begin{thm} \label{ymenergyidentity}(e.g. \cite{RiviequantYM} and \cite[Chapter 4]{geofourfolds})
 Let $ A_k$ be a sequence of Yang-Mills connections in a bundle $E$ with a compact structure group $G$ over $B_1  \subset (M^4,g)$, satisfying
 \begin{equation*}
\int_{B_1}|F_{A_k}|^2dV_g \leq C
\end{equation*}
for some $C>0$. Assume that there is a sequence positive $\lambda_k\rightarrow 0$ such that
$A_k(\lambda_k x)\rightarrow A_{\omega}$
on any compact set $ K \subset \mathbb{R}^4$, and that $A_k$ converges weakly in $W^{ 1,2 }$ to a limit $A_\infty$. Without loss of generality, we assume that $A_\omega$ is the only bubble connection. Then,
 \begin{equation*}
\lim_{k\rightarrow \infty}\int_{B_1}|F_{A_k}|^2dV_g=\int_{B_1}|F_{A_\infty}|^2dV_g+\int_{\mathbb{R}^4}|F_{A_\omega}|^2dx.
\end{equation*}
\end{thm}

\

\section{Yang-Mills energy quantization over non-collapsed
degenerating Einstein 4-manifolds}\label{treeidentity}

\vskip5pt

In this section, we shall prove Theorem \ref{degeymenergyintr} and include some auxiliary results (Definitions, Theorems and Lemmas) when necessary.

\vskip5pt

\begin{proof}[\textbf{Proof of Theorem \ref{degeymenergyintr}}]

\vskip5pt

Let $(M_j ,g_j)$  be a sequence of non-collapsed Einstein 4-manifolds as stated in Theorem \ref{mainconvgethm} and let $A_j$ be a sequence of Yang-Mills connections in bundles $E_j$ over $(M_ j, g_ j)$ with uniformly bounded Yang-Mills $L^2$-energy. By Theorem \ref{mainconvgethm}, without loss of generality, we may assume that there is a subsequence, still denoted by $(M_ j ,g_ j )$, converging to some limit space $(M_ \infty ,g_ \infty )$ which is an Einstein 4-orbifold with only one orbifold singularity $x_a$.

\

\noindent{\bf Case I: Blow-up away from the orbifold singularity}

\vskip5pt

In the situation that the Yang-Mills fields blow up at some point $x_b$ which is away from the orbifold singularity $x_a$, it is easy to see that Yang-Mills energy identity holds, which is stated in Theorem \ref{noneckthm2} in the following. Without loss of generality, we assume that $d_{g_\infty}(x_a,x_b)>1$, then $M_j\supset B_1(x_{b,j})\rightarrow B_1(x_b) \subset M_\infty$ smoothly. The situation that Yang-Mills fields blow up at the orbifold singularity $x_a$ is much more complicated.

To understand the general blow-up picture, we introduce some necessary concepts. We say that
\begin{equation*}
(M_j,g_j, E_j,A_j,x_{a,j})\rightarrow (M_\infty,g_\infty, E_\infty,A_\infty,x_a)
\end{equation*}
smoothly away from the orbifold singularity $x_a$ and a smooth point $x_b$,
if there exists a sequence of into diffeomorphisms $\Phi_ j : M_\infty \setminus\{x_a\}\rightarrow  M _j $ such that:

\

\begin{itemize}
\item[(1)] $ \Phi^*_j g_ j    \rightarrow   g_\infty $ smoothly outside $ x_a$; \\
\item[(2)] $ \Phi^*_j E_ j \rightarrow E_\infty $ smoothly outside $ x_a$ and $x_b$; \\
\item[(3)] $ \Phi^*_j A_ j \rightarrow A_\infty$  smoothly outside $ x_a$ and $x_b$.
\end{itemize}

\

\begin{defn}\label{weakconverconnet}(\textbf{$W^{1,2}$ weak convergence})
For a sequence of (pointed) manifolds $(M_k,g_k,p_k)$ and connections $A_k$ in bundles $E_k$ over $M_k$ and an (pointed) orbifold $(M_\infty,g_\infty,p_\infty)$ equipped with $ E_\infty$ and $A_\infty$, we say that
\begin{equation*}
(M_k,g_k, E_k,A_k,p_k)\rightarrow (M_\infty,g_\infty, E_\infty,A_\infty,p_\infty)
\end{equation*}
weakly in $W^{ 1,2 }$, if $(M_k,g_k,p_k)$ converges smoothly to $(M_\infty,g_\infty,p_\infty)$ outside the set of orbifold singularities $\{y_1,\cdots,y_s\}$, and
 there are an exhaustion $(U_k )$ of $M _\infty\setminus \{y_1,\cdots,y_s\}$, a sequence of diffeomorphisms $\Phi_ k : U_k \rightarrow V_ k \subset M _k$, and a countable set of arbitrarily
small geodesic balls $\{U_{\eta,k}\} $ covering $U_k\setminus\{x_1,\cdots,x_l\}$ ($\{x_1,\cdots,x_l\}$ is the set of the possible concentration points of the Yang-Mills curvature outside the orbifold singularities), $C^\infty$ sections
\begin{equation*}
 A_{\eta,k}\in W^{1,2}(T^1U_{\eta,k}\otimes \mathfrak{g}),\quad \phi_{\eta_1\eta_2,k}\in W^{2,2}(U_{\eta_1,k}\bigcap U_{\eta_2,k},G),
\end{equation*}
and
\begin{equation*}
\tilde{A}_{\eta,k}\in W^{1,2}(T^1V_{\eta,k}\otimes \mathfrak{g}),\quad \tilde{\phi}_{\eta_1\eta_2,k}\in W^{2,2}(V_{\eta_1,k}\bigcap V_{\eta_2,k},G),
\end{equation*}
with $ V_{\eta,k}=\Phi_k(U_{\eta,k})$,
such that:

\

\begin{itemize}
\item[1)]$\Phi_k^{*}\tilde{\phi}_{\eta_1\eta_2,k}\rightharpoonup \phi_{\eta_1\eta_2,k}$ in $ W^{2,2}$;\\
\item[2)] $\Phi_k^{*}\tilde{A}_{\eta,k} \rightharpoonup A_{\eta,k}$ in $ W^{1,2}$;\\
\item[3)] $ A_{\eta_2,k}=\phi_{\eta_1\eta_2,k}^{-1}A_{\eta_1,k}\phi_{\eta_1\eta_2,k}
    +\phi_{\eta_1\eta_2,k}^{-1}d\phi_{\eta_1\eta_2,k}.$\\
\end{itemize}

\end{defn}

\begin{rem}
By standard arguments, we can get $ W^{1,2}$-weak compactness in the above sense for Yang-Mills connections  with bounded Yang-Mills curvature energy in bundles over a sequence of manifolds which converge smoothly outside singularities, see e.g. \cite{UhlenbeckRemovablesingu,UhlenbeckLp,minimizingYM,YMcompbook}.
\end{rem}

\begin{thm}(\textbf{Yang-Mills energy identity away from the orbifold singularity})\label{noneckthm2}
 Let $ A_j$ be a sequence of Yang-Mills connections in bundles $E_j$ over $B_1(x_{b,j}) \subset (M_j,g_j)$ as above, satisfying
 \begin{equation*}
\int_{B_1(x_{b,j})}|F_{A_j}|^2dV_{g_j}\leq C
\end{equation*}
for some $C>0$. Assume that there is a sequence of positive numbers $\lambda_j\rightarrow 0$ such that
\begin{equation*}
\left(B_1(x_{b,j}),(\lambda_j)^2g_j, E_j,A_j,x_{b,j} \right) \rightarrow  \left(\mathbb{R}^4,g_{euc}, E_{\omega_b},A_{\omega_b},0\right)
\end{equation*}
weakly in $W^{ 1,2 }$, and that $A_j$ converges weakly in $W^{ 1,2 }$ to $A_\infty$ in the sense of Definition \ref{weakconverconnet}, where $g_{euc}$ is the Euclidean metric. Moreover, assume that there is only one bubble connection $A_{\omega_b}$. Then
\begin{equation*}
\lim_{j\rightarrow \infty}\int_{B_1(x_{b,j})}|F_{A_j}|^2dV_{g_j}=\int_{B_1(x_b)}|F_{A_\infty}|^2dV_{g_\infty}+\int_{\mathbb{R}^4}|F_{A_{\omega_b}}|^2dx.
\end{equation*}
\end{thm}

\

\noindent{\bf Case II: Blow-up at the orbifold singularity}

\vskip5pt

In the rest part of this section, we shall focus on the more subtle case that the Yang-Mills connections $A_j$ blow up at the orbifold singularity $x_a$. Without loss of generality, we assume that there is only one ALE bubble manifold $M_a$ at this singularity $x_a$. Note that $M_a$ is smooth in this case. The more general cases that ALE bubble orbifolds occur will be studied later.

Firstly, we shall make the decomposition of the underlying domain manifolds $M_j$ at the singularity $x_a$. Recall that we have assumed that $x_a$ is the only orbifold singularity and there is only one ALE bubble manifold.
By Theorem \ref{mainconvgethm},
 for every $j$, there exists $x_{a,j}\in M_j$  and a positive number $r_j\rightarrow 0$ (as $j\rightarrow \infty$) such that

 \

\begin{itemize}
\item[(a)] $B(x_{a,j} ,\delta)$ converges to $B(x_a, \delta) $ in the Gromov-Hausdorff distance for all $\delta>0$.\\

\item[(b)] $\left(( M_j , r_{j}^{-2}g_j),x_{a,j}\right)$ converge to $(( M_a , h_a),x_{a,\infty})$ in the Gromov-Hausdorff distance, where $( M_a , h_a)$ is a complete, noncompact, Ricci flat, non-flat ALE $4$-manifold. There exists an into diffeomorphism $G_j : M_a \rightarrow M_j$ such that $G^{*}_{j}(r_{j}^{-2} g_j)$ converges to $h_a$ in the $C^\infty$-topology on $M_a$.\\

\end{itemize}

For large $R>0$, we call
$$A_{r_j R,\delta}(x_{a,j})=B(x_{a,j},\delta)\setminus B(x_{a,j},r_j R)$$
the degenerate neck region, and $B(x_{a,j},r_j R)$ the ALE bubble domain, and $M_j\setminus B(x_{a,j},\delta)$ the base domain.

Generally, it is possible that there are several bubble connections $A_{\omega^{k}_{a}}$ ($k=1,\cdots ,m_a$) occurred at the orbifold singularity $x_a$.
Suppose that $A_j$ in $E_j$ over $(M_j, (\lambda_{j}^{k})^{-2}g_j, x_{a,k,j})$
converges in the sense of Definition \ref{weakconverconnet} to $A_{\omega^{k}_{a}}$ in some bundle $E_{\omega^{k}_{a}}$ over $(M_{\omega^{k}_{a}}, h_{k,a}, y_{k,a})$ as $j\rightarrow \infty$. We simply write it as
$$\left(M_j, (\lambda_{j}^{k})^{-2}g_j, x_{a,k,j},E_j,A_j \right)\rightarrow (M_{\omega^{k}_{a}}, h_{k,a}, y_{k,a},E_{\omega^{k}_{a}},A_{\omega^{k}_{a}})$$
weakly in $W^{1,2}$, where $\lim_{j\rightarrow \infty}d_{g_j}(x_{a,k,j},x_{a,j}) =0.$

\

Now we investigate the blow-up process in the following three cases:

\

\begin{itemize}
\item[{\bf Case $\alpha$}:] $\lim_{j\rightarrow\infty}\frac{d_{g_j}(x_{a,k,j},x_{a,j})}{ r_j}<\infty$ and $\lim_{j\rightarrow\infty}\frac{\lambda_j^{k}}{ r_j}=0$;\\

\item[{\bf Case $\beta$}:]  $\lim_{j\rightarrow\infty}\frac{d_{g_j}(x_{a,k,j},x_{a,j})}{r_j}<\infty$ and $0<\lim_{j\rightarrow\infty}\frac{\lambda_j^{k}}{ r_j}<\infty$;\\

\item[{\bf Case $\gamma$}:] others.\\
\end{itemize}

Roughly speaking, the domain of the bubble of Case $\beta$ is (almost) the ALE bubble domain, bubbles of Case $\alpha$ are on the bubble of Case $\beta$, the bubble of Case $\beta$ is on some bubble of Case $\gamma$ or is away from bubbles of Case $\gamma$. Without loss of generality, we can assume $x_{a,k,j}=x_{a,j}$ in Case $\beta$.

In the sequel, bubble connections of Case $\alpha$, Case $\beta$, Case $\gamma$ will be denoted by $A_{\omega^{\alpha}}$, $A_{\omega^{\beta}}$, $A_{\omega^{\gamma}}$, respectively. The centers of these three types of bubbles will be denoted by $x^{\alpha}_{a,j}$, $x^{\beta}_{a,j}$, $x^{\gamma}_{a,j}$ respectively, and the three types of bubble scales will be denoted by $\lambda_{j}^{\alpha}$, $\lambda_{j}^{\beta}$, $\lambda_{j}^{\gamma}$, respectively.

In Case $\alpha$,
$$\left(M_j, (\lambda_{j}^{\alpha})^{-2}g_j, x^{\alpha}_{a,j},E_j,A_j \right)\rightarrow \left(\mathbb{R}^4, g_{\text{euc}}, 0,E_{\omega^{\alpha}},A_{\omega^{\alpha}} \right)$$
weakly in $W^{1,2}$. Without loss of generality, we assume there is only one bubble of Case $\alpha$.

In Case $\beta$, by suitable adjustment of the bubble scale, we may assume $r_j=\lambda_{j}^{\beta}$. By our assumption, we have
$$(M_j, (\lambda_{j}^{\beta})^{-2}g_j, x^{\beta}_{a,j},E_j,A_j)\rightarrow ( M_a , h_a,x_{a,\infty},E_{\omega^{\beta}},A_{\omega^{\beta}})$$
weakly in $W^{1,2}$.

\begin{rem}\label{bublleAandD}
By above construction, the bubble $A_{\omega^{\alpha}}$ lies on the bubble $A_{\omega^{\beta}}$, by the same arguments as for Theorem \ref{noneckthm2} we know that there is no Yang-Mills energy loss in the following neck region
$$A_{\lambda_j^{\alpha} R,\delta r_j}(x^{\alpha}_{a,j})\equiv B(x^{\alpha}_{a,j},{\delta r_j}) \setminus B(x^{\alpha}_{a,j},{\lambda_j^{\alpha} R}).$$
\end{rem}

\begin{que}\label{ques1}
How the connection $A_j$ behaves on the degenerate neck region
$$A_{r_jR,\delta}(x_{a,j})=B(x_{a,j},\delta)\setminus B(x_{a,j},r_jR)$$
when $j\rightarrow\infty$?
\end{que}

The above question is closely related to the analysis of bubble connections of Case $\gamma$.

Firstly, we shall decompose the degenerate neck region in Question \ref{ques1} using bubbles of Case
$\gamma$. We refer to \cite{zhu2010harmonic,lzMem} for analogous situations in the cases of harmonic maps and $\alpha$-harmonic maps from degenerating Riemann surfaces.

\

\noindent{\bf Bubble and neck decomposition in the degenerate region}:

\vskip5pt

Set $T_j=\log(r_jR)$ and $T(\delta)=\log(\delta)$.

\begin{prop}\label{properneckc2}
Notations and assumptions are as above.
\begin{itemize}
\item [(1)] "Asymptotic boundary conditions":
\begin{equation*}
\lim_{j\rightarrow \infty}\omega(A_j,P_{T_j,T_j+L})=\lim_{\delta\rightarrow 0}\omega(A_j,P_{T(\delta)-L,T(\delta)})=0, \quad \forall L\geq 1,
\end{equation*}
where $P_{T_j^{1},T_j^{2}}$ is the cylinder corresponding to $A_{e^{T_j^{1}},e^{T_j^{2}}}(x_{a,j})\subset M_j$,
\begin{equation*}
\omega(A_j,P_{T_j^{1},T_j^{2}})\equiv \sup_{t\in[T_j^{1},T_j^{2}-1]}\int_{A_{e^{t},e^{t+1}}(x_{a,j})}|F_{A_j}|^2 dV_{g_j}.
\end{equation*}

 \item [(2)] "bubble domain and neck domain": after selection of a subsequence, which is still denoted by $A_j$, either
\begin{itemize}

\item [(2.1)]
\begin{equation*}
\lim_{\delta\rightarrow 0}\lim_{R\rightarrow \infty}\lim_{j\rightarrow \infty}\omega(A_j,P_{T_j,T(\delta)})=0,
\end{equation*}

or

\item [(2.2)] there exists $N_2>0$ independent of $j$ and $2N_2$ sequence $\{a_{j}^1\},\{b_{j}^1\}$, $\cdots, \{a_{j}^{N_2}\},\{b_{j}^{N_2}\}$ such that
\begin{equation*}
T_j\leq a_j^{1}\ll b_j^{1}\leq \cdots \leq a_j^{N_2}\ll b_j^{N_2}\leq T(\delta) \quad (a_j^{\mathfrak{K}}\ll b_j^{\mathfrak{K}} \, \ \text{means}\,\lim_{j\rightarrow \infty}b_j^{\mathfrak{K}}- a_j^{\mathfrak{K}}=\infty )
\end{equation*}
and
\begin{equation*}
 |b_j^{\mathfrak{K}}- a_j^{\mathfrak{K}}|\ll |T_j|,\quad \text{i.e.}\ \lim_{j\rightarrow \infty} \frac{|b_j^{\mathfrak{K}}- a_j^{\mathfrak{K}}|}{|T_j|}=0.
\end{equation*}
\end{itemize}
Denote
 $$J_j^{\mathfrak{K}}\equiv P_{a_j^{\mathfrak{K}}, b_j^{\mathfrak{K}}}, \quad \mathfrak{K}=1,\cdots, N_2,$$
 $$I_{j}^0=P_{T_j, a_{j}^1}, \quad I_j^{N_2}=P_{b_j^{N_2},T(\delta)}, \quad I_j^{\mathfrak{K}}=P_{b_j^{\mathfrak{K}},a_j^{\mathfrak{K}+1}}, \quad  \mathfrak{K}=1,\cdots, N_2-1. $$
 Then
\begin{itemize}
  \item [({\romannumeral1})]
  $\forall \mathfrak{K}=0,1,\cdots, N_2,$ $\lim_{\delta\rightarrow 0}\lim_{R\rightarrow \infty}\lim_{j\rightarrow \infty}\omega(A_j,I_j^{\mathfrak{K}})=0$.

\item [({\romannumeral2})] $\forall \mathfrak{K}=1,\cdots, N_2,$ there is a bubble tree that consists of at most finitely many bubble connections, i.e., finite energy Yang-Mills connections in some bundles over $\mathbb{R}^4/\Gamma$. Here, for simplicity of notations, we assume there is only one such bubble connection $A_{\omega^{\gamma,\mathfrak{K}}}$ such that
 \begin{equation*}
 \lim_{\delta\rightarrow 0}\lim_{R\rightarrow \infty}\lim_{j\rightarrow \infty}\mathcal{YM}(A_j,\bar{J}_j^{\mathfrak{K}})=\mathcal{YM}(A_{\omega^{\gamma,\mathfrak{K}}}),
\end{equation*}
 where $\bar{J}_j^{\mathfrak{K}}$ is the region $A_{e^{a_j^{\mathfrak{K}}},e^{b_j^{\mathfrak{K}}}}(x_{a,j}) \subset M_j $ corresponding to $J_j^{\mathfrak{K}}$ and
  \begin{equation*}
\mathcal{YM}(A_j,\bar{J}_j^{\mathfrak{K}})\equiv\int_{\bar{J}_j^{\mathfrak{K}}}|F_{A_j}|^2 dV_{g_j}.
\end{equation*}

\end{itemize}
\end{itemize}
\end{prop}

\begin{proof}
The conclusion (1) follows immediately from our assumptions. The conclusion (2) can be proved in the same spirit as in \cite{zhu2010harmonic}.

For ({\romannumeral1}), we can apply the arguments as in the proof of Proposition 3.1 in \cite{zhu2010harmonic}.

For ({\romannumeral2}), it is sufficient to verify that $A_{\omega^{\gamma,\mathfrak{K}}}$ is a connection in some bundle over $\mathbb{R}^4/\Gamma$.

Theorem \ref{mainconvgethm} says that the degenerate neck region
$$A_{r_jR,\delta}(x_{a,j})=B(x_{a,j},\delta)\setminus B(x_{a,j},r_jR)$$
looks like a potion of a flat cone $\mathbb{R}^4/\Gamma$ for large $j$, moreover, there holds
\begin{equation*}\label{neckcurvturestm}
r^2|Rm(g_j)|\leq C_7 \max \left\{\left(\frac{r_j}{r}\right)^{\varepsilon_5},\left(\frac{r}{r_\infty}\right)^{\varepsilon_5} \right\},
\end{equation*}
 where $C_7>0$, $\varepsilon_5>0$ and $r_\infty>0$ are constants as in Theorem \ref{mainconvgethm}.
Noticing that $a_j^{\mathfrak{K}}\ll b_j^{\mathfrak{K}}$, by choosing the scale $\lambda_j^{\gamma,\mathfrak{K}}\equiv e^{\frac{a_j^{\mathfrak{K}}-b_j^{\mathfrak{K}}}{2}}$, $\left(\bar{J}_j^{\mathfrak{K}}, (\lambda_j^{\gamma,\mathfrak{K}})^{-2}g_j \right)$ converges to $\mathbb{R}^4/\Gamma$ as $j\rightarrow \infty$, and $A_j$ converges to $A_{\omega^{\gamma,\mathfrak{K}}}$.
\end{proof}

Next, we are going to construct the whole bubble tree for the convergence of Yang-Mills connections in the case in which there are several ALE bubble manifolds (orbifolds) emerging (at the same orbifold singularity) from the Gromov-Hausdorff convergence of non-collapsing Einstein 4-manifolds.

From the analysis discussed above, we know that three types of bubble connections can possibly occur, namely, $A_{\omega^{\alpha}}$ in some bundle over $\mathbb{R}^4$, $A_{\omega^{\beta}}$ in some bundle over some $ M_{ALE}$, and $A_{\omega^{\gamma}}$ in some bundle over $\mathbb{R}^4/\Gamma$, where $ M_{ALE}$ represents one of the ALE bubble manifolds (orbifolds) in the tree $\mathbf{Tr_{ALE}}$ of bubble tree convergence for non-collapsed Einstein 4-manifolds in Theorem \ref{mainconvgethm}.

Firstly, bubble connections of type $A_{\omega^{\beta}}$ occur in the region where the Riemannian curvature concentrate, and the number of this kind of bubble connections is equal to the number of ALE bubble manifolds (orbifolds) in the tree $\mathbf{Tr_{ALE}}$. Secondly, bubble connections of type $A_{\omega^{\gamma}}$ appear in degenerate neck regions which connect those ALE bubble spaces in the tree $\mathbf{Tr_{ALE}}$. Thirdly, bubble maps of type $A_{\omega^{\alpha}}$ appear where there is no degeneration of Einstein metrics at its blow-up scale, namely
$$\left(M_j, (\lambda_{j}^{\alpha})^{-2}g_j, x^{\alpha}_{a,j}\right) \rightarrow \left(\mathbb{R}^4, g_{\text{euc}}, 0 \right)$$
as $j\rightarrow \infty$. In other words, bubble connections of type $A_{\omega^{\alpha}}$ appear at points in the smooth regions in those ALE bubble spaces in the tree $\mathbf{Tr_{ALE}}$.
With these facts in mind, one can easily construct the whole bubble tree by induction. We remark here that the whole process will be terminated in finite steps, since the number of non-trivial ALE bubble spaces and the number of non-flat bubble connections over $\mathbb{R}^4$ or $\mathbb{R}^4/\Gamma$ must be finite, due to the energy gap Theorem \ref{energygap} and the finite total energy assumption.

By above discussion, in order to prove Theorem \ref{degeymenergyintr}, we only need to prove the following:

\begin{thm}\label{degenockenergy}
Let $\bar{I}_j^{\mathfrak{K}}=A_{e^{b_j^{\mathfrak{K}}},e^{a_j^{\mathfrak{K}+1}}}(x_{a,j})\subset M_j $ be the annulus region in $M_j$ corresponding to $I_j^{\mathfrak{K}}$ in Proposition \ref{properneckc2}, then
\begin{equation*}\label{energyidentityformula}
\lim_{\delta\rightarrow 0} \lim_{R\rightarrow \infty}\lim_{j \rightarrow \infty}\int_{\bar{I}_j^{\mathfrak{K}}}|F_{A_j}|^2 dV_{g_j}=0.
\end{equation*}
\end{thm}

It is interesting to ask the following question, although it does not make difference to the validity of Theorem \ref{degeymenergyintr} whether or not there is a positive answer.
\begin{que}
Is there an energy gap for Yang-Mills connections in bundles over the ALE spaces in the above?
\end{que}

To proceed with the proof of Theorem \ref{degenockenergy}, we recall and present some useful estimates.

Firstly, we prove two key differential inequalities for Yang-Mills fields by exploring the fine structure of the Yang-Mills equation.

\begin{lem}\label{katoinequ}(Lemma 4.8 of \cite{Nakajima1990}, Lemma 3.1 of \cite{decayParker})
Let $A$ be a Yang-Mills connection in a bundle $E$ over a Riemannian 4-manifolds $(M,g)$,
for $\delta = \frac{1}{2}$, it holds that
\begin{eqnarray*}\label{ymcurvnormequ}
|\nabla_{g,A}F_{A}|^2\geq(1+\delta)\big{|}\nabla_{g}|F_{A}|\big{|}^2.
\end{eqnarray*}
\end{lem}
Recently in \cite{Uhlenbkato}, Smith-Uhlenbeck derived a quite non-trivial Kato-Yau type inequality of the form
\begin{eqnarray*}
\frac{n}{n-1}\big{|}\nabla_{g}|F_{A}|\big{|}^2\leq|\nabla_{g,A}F_{A}|^2
\end{eqnarray*}
for Yang-Mills connections in vector bundles over Riemannian $n$-manifolds with $n\geq4$, and applied it to reprove the
higher dimensional removable singularities theorem in \cite{taotian}.
\

\begin{lem}\label{keyymequ}
Let $A_j$ be a Yang-Mills connection on $(M_j,g_j)$ and let $\Delta_j=-\nabla^*_{g_j}\nabla_{g_j} $ be the usual Laplace operator acting on functions, then we have
\begin{eqnarray}\label{ymcurvnormequ201}
\triangle_j|F_{A_j}|\geq
-C_{ym}|F_{A_j}|\left(| F_{A_j}|+|Rm(g_j)|\right),
\end{eqnarray}
and
\begin{eqnarray*}
\triangle_j|F_{A_j}|^{1/2}
\geq-C_{ym}|F_{A_j}|^{1/2}\left(| F_{A_j}|+|Rm(g_j)|\right).
\end{eqnarray*}
for some universal constant $C_{ym}>0$.
\end{lem}

\begin{proof}
Recall the following Bochner-Weitzenb\"{o}ck formula for 2-forms in \cite{BourgLawsonYM}:
\begin{eqnarray*}
\triangle_{g,A}\Phi =(D_A^{*}D_A+D_AD_A^{*})\Phi=\nabla^*_{g,A}\nabla_{g,A}\Phi+F_{A}\sharp\Phi+Rm\sharp \Phi,
\end{eqnarray*}
where $\sharp$ denotes the contraction of tensors. For Yang-Mills curvature forms, there holds
\begin{eqnarray*}
\triangle_{g,A}F_A=0.
\end{eqnarray*}
Therefore we have
\begin{eqnarray*}
-\nabla^*_{g_j,A_j}\nabla_{g_j,A_j}F_{A_j}=F_{A_j}\sharp F_{A_j}+Rm\sharp F_{A_j}.
\end{eqnarray*}
On one side,
\begin{eqnarray*}
\frac{1}{2}\triangle_j|F_{A_j}|^{2}&=&\left\langle -\nabla^*_{g_j,A_j}\nabla_{g_j,A_j}F_{A_j} ,F_{A_j}\right\rangle+\langle\nabla_{g_j,A_j}F_{A_j}, \nabla_{g_j,A_j}F_{A_j}\rangle\\
&=&\left\langle F_{A_j}\sharp F_{A_j}+Rm(g_j)\sharp F_{A_j} ,F_{A_j}\right\rangle+|\nabla_{g_j,A_j}F_{A_j}|^2.
\end{eqnarray*}
On the other side,
\begin{eqnarray*}
\frac{1}{2}\triangle_j|F_{A_j}|^2
=\big{|}\nabla_{g_j}|F_{A_j}|\big{|}^2+|F_{A_j}| \cdot \triangle_j|F_{A_j}|.
\end{eqnarray*}
So it follows that
\begin{eqnarray}\label{ymcurvnormequ11}
|F_{A_j}| \cdot \triangle_j|F_{A_j}|\nonumber=\left\langle F_{A_j}\sharp F_{A_j}+Rm(g_j)\sharp F_{A_j} ,F_{A_j}\right\rangle+\left(|\nabla_{g_j,A_j}F_{A_j}|^2-\big{|}\nabla_{g_j}|F_{A_j}|\big{|}^2\right).
\end{eqnarray}
Therefore, we have
\begin{eqnarray*}
\triangle_j|F_{A_j}|\geq
-C|F_{A_j}|\left(| F_{A_j}|+|Rm(g_j)|\right),
\end{eqnarray*}
where $C>0$ is some universal constant.

Next, by direct computations with the use of (\ref{ymcurvnormequ11}), we have
\begin{eqnarray*}
& &\triangle_j|F_{A_j}|^{(1-\delta)}=
-\delta(1-\delta)|F_{A_j}|^{(-1-\delta)}\big{|}\nabla_{g_j}|F_{A_j}|\big{|}^2
+(1-\delta)|F_{A_j}|^{-\delta}\triangle_j|F_{A_j}|\\
&=&|F_{A_j}|^{(-1-\delta)}\left[-\delta(1-\delta)\big{|}\nabla_{g_j}|F_{A_j}|\big{|}^2
+(1-\delta)\big{(}|\nabla_{g_j,A_j}F_{A_j}|^2-\big{|}\nabla_{g_j}|F_{A_j}|\big{|}^2\big{)}\right]\\
& &+|F_{A_j}|^{(-1-\delta)}\left\langle F_{A_j}\sharp F_{A_j}+Rm(g_j)\sharp F_{A_j} ,F_{A_j}\right\rangle\\
&\geq&(1-\delta)|F_{A_j}|^{(-1-\delta)}\left[-(1+\delta)\big{|}\nabla_{g_j}|F_{A_j}|\big{|}^2
+|\nabla_{g_j,A_j}F_{A_j}|^2\right]\\
& &-C|F_{A_j}|^{(1-\delta)}\left(| F_{A_j}|+|Rm(g_j)|\right),
\end{eqnarray*}
where $C>0$ is some universal constant.

Therefore, by Lemma \ref{katoinequ}, if $A_j$ is a Yang-Mills connection, then
\begin{eqnarray*}\label{seldualineq}
\triangle_j|F_{A_j}|^{1/2}
\geq-C|F_{A_j}|^{1/2}\left(| F_{A_j}|+|Rm(g_j)|\right),
\end{eqnarray*}
where $C>0$ is some universal constant.
\end{proof}

Secondly, we recall some analytical results about elliptic estimates. We refer to \cite{bando1990bubbling, BKN} for more details. Let $M$ be a complete $n$-dimensional ($n\geq3$) Riemannian manifold with a fixed point $o\in M$, $A_{r_1,r_2}=B(o,r_2)\setminus B(o,r_1) \subset M.$ Assume that there is a domain $A_{r_0,r_\infty}$ in $M$ with $0\leq r_0<r_\infty$ which satisfies
\begin{equation*}
\big{(}\int_{A_{r_0,r_\infty}}|v|^{2\gamma}dV_g\big{)}^{1/\gamma}\leq S\int_{A_{r_0,r_\infty}} |\nabla v|^2 dV_g, \quad \text{  for all  } v\in C_{c}^1(A_{r_0,r_\infty}),
\end{equation*}
\begin{equation*}
\text{vol}(A_{r_1,r_2})\leq V, \quad  \text{  for all  }r_0\leq r_1\leq r_2 \leq r_\infty
\end{equation*}
with some positive constants $S>0,V>0$ and $\gamma=\frac{n}{n-2}$.
Let $u$ be a non-negative function defined on $A_{r_0,r_\infty}$ which satisfies
\begin{equation*}
\triangle u\geq- f u, \quad \text{   on    } A_{r_0,r_\infty}
\end{equation*}
for a non-negative function $f$. Then we have the following estimates in Lemma \ref{ellestlem1}, Lemma \ref{ellestlem2} and Lemma \ref{c0normest0}.

\begin{rem}\label{sobolevrem}
By arguments in Section 2 of \cite{anderson1989ricci}, there is an uniform upper bound for the Sobolev constants $S$ of closed manifolds with bounded Ricci curvatures and bounded diameters and satisfying the volume non-collapsing condition as in Theorem \ref{mainconvgethm}.
\end{rem}

\begin{lem}(\cite{bando1990bubbling})\label{ellestlem1}
Suppose $f\in L^{n/2},$ and $u\in L^p$ for some $p\in[p_0,p_1] $ where $p_0>1$. Then $u\in L^q$ for all $q\geq p$, and exists $\varepsilon_1(S,V,p_0,p_1) $ such that if
\begin{equation*}
\int_{A_{r,8r}} f^{n/2}dV_g\leq \varepsilon_1 \text{  with  } r_0\leq r<8 r \leq r_\infty,
\end{equation*}
then we have
 \begin{equation*}
\big{(}\int_{A_{2r,4r}} u^{p\gamma}dV_g\big{)}^{1/\gamma}\leq C_1r^{-2}\int_{A_{r,8r}} u^{p}dV_g,
\end{equation*}
where $C_1=C_1(S,V,p_0)>0.$
\end{lem}

\begin{lem}(\cite{bando1990bubbling})\label{ellestlem2}
Suppose $f\in L^{n/2}$, and $u\in L^{p}$ for some $p\in [p_0,p_1]$ where $p_0>\gamma$. Then there exists $\varepsilon_2=\varepsilon_2(S,V,p_0,p_1)>0$ such
that for $r_0\leq r_1<2r_1<r_2<2r_2\leq r_\infty$, if
\begin{equation*}
\int_{A_{r_1,2r_2}} f^{n/2}dV_g\leq \varepsilon_2 ,
\end{equation*}
then
\begin{equation*}
\int_{A_{2r_1,r_2}} u^{p}dV_g\leq C_2 \int_{A_{r_1,2r_1}\bigcup A_{r_2,2r_2}} u^{p}dV_g,
\end{equation*}
and
\begin{equation*}
\int_{A_{2r_1,r_2}} u^{p}dV_g\leq C_2 \max\left\{\left(\frac{r_0}{r_1}\right)^{\varepsilon_3}, \left(\frac{r_2}{r_\infty}\right)^{\varepsilon_3}\right \}\int_{ A_{r_0,r_\infty}} u^{p}dV_g,
\end{equation*}
where $C_2=C_2(S,V,p_0)>0$ and $\varepsilon_3=\varepsilon_3(S,V,p_0)>0$.
\end{lem}

\begin{lem}(\cite{bando1990bubbling})\label{c0normest0}
If $u\in L^{p}$ for some $p>1$, and if for some $q>n/2$
\begin{equation}\label{coeffdecay}
\int_{A_{r,8r}} f^{q}dV_g\leq C r^{-(2q-n)} ,
\end{equation}
with some constant $C>0$ for any $r$ such that $r_0\leq r<8r \leq r_\infty$, then we have
\begin{equation*}
\sup_{A_{2r,4r}}u^{p}\leq C_3 r^{-n}\int_{A_{r,8r}} u^{p}dV_g ,
\end{equation*}
where $C_3=C_3(C,S,V,p,q)>0$.
\end{lem}

Now, we recall some vital results for the proof of Theorem \ref{degenockenergy}.
\begin{itemize}

\item[($\mathcal{V}_1$)]

By Theorem \ref{mainconvgethm}, there exist positive constants $C_7>0$ and $\varepsilon_5>0$ such that for $ 4r_j\leq r< 4r \leq r_\infty$ it holds that
\begin{equation*}\label{neckcurvturestmprop}
r^2|Rm(g_j)|\leq C_7 \max \left\{\left(\frac{r_j}{r}\right)^{\varepsilon_5},\left(\frac{r}{r_\infty}\right)^{\varepsilon_5}  \right\}.
\end{equation*}

\item[($\mathcal{V}_2$)]

 By Proposition \ref{properneckc2},
  $\forall  \ \varepsilon>0$, $\exists  \ R_{\varepsilon}>0,\delta_\varepsilon>0, j_\varepsilon>0$, such that if $R> R_{\varepsilon},\delta<\delta_\varepsilon$, $j>j_\varepsilon$, and $$A_{r,8r}(x_{a,j})\subset\bar{I}_j^{\mathfrak{K}}\subset A_{r_jR,\delta}(x_{a,j})\subset A_{r_j,r_\infty}(x_{a,j}),$$ then
\begin{equation}\label{ymenergysmall3.18}
\int_{A_{r,8r}(x_{a,j})}|F_{A_j}|^2 dV_{g_j}<\varepsilon^2.
\end{equation}

\end{itemize}

Now we start to prove Theorem \ref{degenockenergy}.
Firstly, we show that
for $A_{r,8r}(x_{a,j})\subset \bar{I}_j^{\mathfrak{K}},$ by choosing the constant $\varepsilon>0$ in (\ref{ymenergysmall3.18}) to be small enough, we have
\begin{equation}\label{connectionlpest}
\left(\int_{A_{2r,4r}(x_{a,j})}|F_{A_j}|^4 dV_{g_j}\right)^{1/2}<C_1r^{-2}\int_{A_{r,8r}(x_{a,j})}|F_{A_j}|^2 dV_{g_j},
\end{equation}
where $C_1>0$ is as  in Lemma \ref{ellestlem1}.

By applying Lemma \ref{ellestlem1} with $n=4$, $\gamma=2$. By assumption that there is no curvature concentration on $\bar{I}_j^{\mathfrak{K}}$, see ($\mathcal{V}_1$) and ($\mathcal{V}_2$) in the above. So
$$u=|F_{A_j}|\in L^2(\bar{I}_j^{\mathfrak{K}})$$ and
$$f=C_{ym}(|F_{A_j}|+|Rm(g_j)|)\in L^2(\bar{I}_j^{\mathfrak{K}}).$$
Moreover, it holds that for $A_{r,8r}(x_{a,j})\subset\bar{I}_j^{\mathfrak{K}}$, by choosing $\varepsilon>0$ small enough in (\ref{ymenergysmall3.18}), then
\begin{equation*}
\int_{A_{r,8r}(x_{a,j})} f^{2}dV_{g_j}<\varepsilon_1,
\end{equation*}
where $\varepsilon_1>0$ is as in Lemma \ref{ellestlem1}.
Therefore, the estimate \eqref{connectionlpest} follows by applying Lemma \ref{ellestlem1} to (\ref{ymcurvnormequ201}) which states that
\begin{eqnarray*}\label{ymcurvnormequ20}
\triangle_j|F_{A_j}|\geq
-C_{ym}|F_{A_j}|\left(| F_{A_j}|+|Rm(g_j)|\right),
\end{eqnarray*}
where $C_{ym}>0$ is the universal constant in Lemma \ref{keyymequ}.

Secondly, we show that
for $A_{r,8r}(x_{a,j})\subset \bar{I}_j^{\mathfrak{K}},$ by choosing the constant $\varepsilon>0$ in (\ref{ymenergysmall3.18}) to be small enough, we have
\begin{equation}\label{decayYM}
\int_{A_{2r,4r}(x_{a,j})} |F_{A_j}|^2dV_{g_j}\leq C_2 \max\left\{\left(\frac{e^{b_j^{\mathfrak{K}}}}{r}\right)^{\varepsilon_3}, \left(\frac{r}{e^{a_j^{\mathfrak{K}+1}}}\right)^{\varepsilon_3} \right\}\int_{\bar{I}_j^{\mathfrak{K}}} |F_{A_j}|^2dV_{g_j},
\end{equation}
where $C_2>0$ and $\varepsilon_3>0$ are as in Lemma \ref{ellestlem2}.

By Lemma \ref{keyymequ},
\begin{eqnarray*}
\triangle_j|F_{A_j}|^{1/2}
\geq-C_{ym}|F_{A_j}|^{1/2}\left(| F_{A_j}|+|Rm(g_j)|\right).
\end{eqnarray*}
As before, we know that $f=C_{ym}(| F_{A_j}|+|Rm(g_j)|)\in L^2(\bar{I}_j^{\mathfrak{K}}), $ and $u=|F_{A_j}|^{1/2}\in L^4(\bar{I}_j^{\mathfrak{K}}).$

By choosing the constant $\varepsilon>0$ in (\ref{ymenergysmall3.18}) small enough, we have
\begin{equation*}
\int_{A_{r,8r}(x_{a,j})} f^{2}dV_{g_j}\leq \varepsilon_2
\end{equation*}
where $\varepsilon_2>0$ is as in Lemma \ref{ellestlem2}.
Noticing that $p=4>\gamma=2$, we get the desired estimate \eqref{decayYM} by Lemma \ref{ellestlem2}.

Thirdly, we show that
for $A_{r,8r}(x_{a,j})\subset \bar{I}_j^{\mathfrak{K}}, $
\begin{equation}\label{keydecayYM}
\sup_{A_{2r,4r}(x_{a,j})}|F_{A_j}|^2\leq C_4 r^{-4}\max\left\{\left(\frac{e^{b_j^{\mathfrak{K}}}}{r}\right)^{\varepsilon_3}, \left(\frac{r}{e^{a_j^{\mathfrak{K}+1}}}\right)^{\varepsilon_3} \right\} \int_{\bar{I}_j^{\mathfrak{K}}} |F_{A_j}|^2dV_{g_j},
\end{equation}
where $C_4=C_4(S,V)>0$, and $\varepsilon_3>0$ is as in Lemma \ref{c0normest0}.

By \eqref{connectionlpest}, the two facts ($\mathcal{V}_1$) and ($\mathcal{V}_2$),
$f=C_{ym}(| F_{A_j}|+|Rm(g_j)|) $ satisfies (\ref{coeffdecay}) with $q=4$ for some universal constant $C>0$, and
$$u=|F_{A_j}|\in L^2 \left(\bar{I}_j^{\mathfrak{K}} \right).$$
Then by Lemma \ref{c0normest0}, \eqref{decayYM}, and the differential inequality (\ref{ymcurvnormequ201}) which states that
\begin{eqnarray*}\label{ymcurvnormequ20}
\triangle_j|F_{A_j}|\geq
-C_{ym}|F_{A_j}|\left(| F_{A_j}|+|Rm(g_j)|\right),
\end{eqnarray*}
we can obtain the desired decay estimate \eqref{keydecayYM}.

Finally, by applying \eqref{keydecayYM}, we can prove Theorem \ref{degenockenergy}. Then by previous arguments, Theorem \ref{degeymenergyintr} follows immediately. This completes the proof of Theorem \ref{degeymenergyintr}.
\end{proof}

Now we state a higher dimensional analog of Theorem \ref{degeymenergyintr}.

\begin{thm}\label{degeymenergyhd} Let $n>4$. Let $(M_j,g_j)$ be a sequence of (real) $n$ dimensional Einstein manifolds with uniformly bounded Einstein constants and diameters, uniform $L^{\frac{n}{2}}$ Riemannian curvature bounds, and positive uniform volume lower bounds. Let $E_j$ be a sequence of vector bundle over $M_j$, and $A_j$ a sequence of Yang-Mills connection in $E_j$ with uniform $L^{\frac{n}{2}}$ Yang-Mills curvature bounds. Then up to a subsequence, we have the following identity:

\begin{eqnarray*}
\lim_{j\rightarrow \infty}\int_{M_j}|F_{A_j}|^{\frac{n}{2}} dV_{g_j} &=& \int_{M_\infty}|F_{A_\infty}|^{\frac{n}{2}}dV_{g_\infty} +\sum_{l=1}^{v}\sum_{k=1}^{N_{b_l}}\mathcal{YM}_n(\mathbb{R}^n, A_{\omega_{b_l,k}})  \\
&& +\sum_{l=1}^{u}\sum_{\eta=\alpha,\beta,\gamma}\sum_{k=1}^{N_{a_l,\eta}} \mathcal{YM}_n(X_{\eta,k,a_l}, A_{\omega^{a_l,\eta,k}}),
\end{eqnarray*}
where $\mathcal{YM}_n(X, A)$ is the $L^{\frac{n}{2}}$ Yang-Mills energy of the connection $A$ over the underlying space $X$. Here $X_{\alpha,k,a_l}$ is $\mathbb{R}^n$, $X_{\beta,k,a_l}$ is a Ricci flat ALE bubble space and $X_{\gamma,k,a_l}$ is of the form $\mathbb{R}^n/\Gamma$ for some nontrivial finite group $\Gamma \subset SO(n)$.
\end{thm}

\begin{rem}\label{hgdimrem}
Notice that bubble tree convergence result similar as in Theorem \ref{mainconvgethm} hold for higher dimensional Einstein manifolds. Since
\begin{eqnarray*}
(| F_{A_j}|+|Rm(g_j)|)&\in& L^{n/2} \left(\bar{I}_j^{\mathfrak{K}} \right)\\
| F_{A_j}|\in L^{p} \left(\bar{I}_j^{\mathfrak{K}} \right),\quad p=n/2&>&\frac{n}{n-2}\quad \text{when} \,\, n>4,
\end{eqnarray*}
by applying Lemma \ref{ellestlem1}, Lemma \ref{ellestlem2} and Lemma \ref{c0normest0} to the inequality
\begin{eqnarray*}
\triangle_j|F_{A_j}|\geq
-C_{ym}|F_{A_j}|\left(| F_{A_j}|+|Rm(g_j)|\right)
\end{eqnarray*}
which is valid in higher dimensions $n>4$, we can get a similar estimate for Yang-Mills curvatures of \eqref{keydecayYM} in the higher dimensional case, therefore Theorem \ref{degeymenergyhd} holds. Then a result about quantization of Euler numbers of holomorphic vector bundles over a sequence of higher dimensional K\"{a}hler-Einstein manifolds analogous to
Theorem \ref{appli33} follows immediately.
\end{rem}

\vskip10pt

%\noindent {\bf Data availability:} The manuscript has no associated data.

% \vskip10pt

%\noindent {\bf Conflicts of interests: }  The authors have no conflicts of interests to disclose.

\
%\bibliographystyle{plain}
%\bibliography{foo}

\end{document}